\documentclass[a4paper,11pt]{article}
\usepackage[a4paper]{geometry}

\usepackage[backend = bibtex]{biblatex}
\usepackage{graphicx}%
\usepackage{multirow}%
\usepackage{amsmath,amssymb,amsfonts}%
\usepackage{amsthm}%
\usepackage{mathrsfs}%
\usepackage[title]{appendix}%
\usepackage{xcolor}%
\usepackage{textcomp}%
\usepackage{manyfoot}%
\usepackage{booktabs}%
\usepackage{algorithm}%
\usepackage{algorithmicx}%
\usepackage{algpseudocode}%
\usepackage{listings}%
\usepackage{mathtools}%

\theoremstyle{plain}%
\newtheorem{theorem}{Theorem}%
\newtheorem{corollary}{Corollary}%
\newtheorem{lemma}{Lemma}%
\newtheorem{proposition}{Proposition}%

\theoremstyle{definition}%
\newtheorem{definition}{Definition}%
\newtheorem{example}{Example}%
\newtheorem{remark}{Remark}%

\numberwithin{theorem}{section}%
\numberwithin{corollary}{section}%
\numberwithin{lemma}{section}%
\numberwithin{definition}{section}%
\numberwithin{example}{section}%
\numberwithin{remark}{section}%
\numberwithin{proposition}{section}%

\newcommand{\smallspace}{\,}%
\newcommand{\mediumspace}{~}%
\newcommand{\largespace}{\quad}%

\newcommand{\productspacing}{\,}%
\newcommand{\delimiterspacing}{\,}%

\newcommand{\period}{\smallspace .}%
\newcommand{\comma}{\smallspace ,}%

\newcommand{\emphasis}[1]{\textit{#1}}

\makeatletter%

\DeclarePairedDelimiterX{\@of}[1]{(}{)}{#1}%
\newcommand{\of}{\@of*}%

\DeclarePairedDelimiterX{\@norm}[1]{\lVert}{\rVert}{#1}%
\newcommand{\norm}{\@norm*}%

\DeclarePairedDelimiterX{\@abs}[1]{\lvert}{\rvert}{#1}%
\newcommand{\abs}{\@abs*}%

\DeclarePairedDelimiterX{\@bra}[1]{[}{]}{#1}%
\newcommand{\bra}{\@bra*}%

\DeclarePairedDelimiterX{\@sset}[1]{\{}{\}}{#1}%
\newcommand{\sset}{\@sset*}%

\DeclarePairedDelimiterX{\@set}[2]{\{}{\}}{#1 \smallspace \delimsize\vert \smallspace #2}%
\newcommand{\set}{\@set*}%

\DeclarePairedDelimiterX{\@cond}[2]{[}{]}{#1 \smallspace \delimsize\vert \smallspace #2}%
\newcommand{\cond}{\@cond*}%

\DeclarePairedDelimiterX{\@sca}[2]{\langle}{\rangle}{#1 \smallspace \delimsize\vert \smallspace #2}%
\newcommand{\sca}{\@sca*}%

\makeatletter%

\newcommand{\too}[2][]{\xrightarrow[#1]{#2}}%
\newcommand{\diffelement}{\mathrm{d}}%
\newcommand{\dd}[2][]{\smallspace \diffelement #1 #2}%

\newcommand{\pardiff}[2][]{\mathchoice{\frac{\partial^{#1}}{\partial #2^{#1}}}{\partial^{#1}_{#2}}{\partial^{#1}_{#2}}{\partial^{#1}_{#2}}}%
\newcommand{\pardiffQ}[3][]{\mathchoice{\frac{\partial^{#1} #2}{\partial #3^{#1}}}{\partial^{#1}_{#3} #2}{\partial^{#1}_{#3} #2}{\partial^{#1}_{#3} #2}}%

\newcommand{\diff}[2][]{\mathchoice{\frac{\diffelement^{#1}}{\diffelement #2^{#1}}}{\diffelement^{#1}_{#2}}{\diffelement^{#1}_{#2}}{\diffelement^{#1}_{#2}}}%
\newcommand{\diffQ}[3][]{\mathchoice{\frac{\diffelement^{#1} #2}{\diffelement #3^{#1}}}{\diffelement^{#1}_{#3} #2}{\diffelement^{#1}_{#3} #2}{\diffelement^{#1}_{#3} #2}}%

\newcommand{\e}[2][]{\exp_{#1}\of{#2}}
\newcommand{\R}{\mathbb{R}}%
\newcommand{\N}{\mathbb{N}}%
\newcommand{\Pb}{\mathbb{P}}%
\newcommand{\E}{\mathbb{E}}%
\newcommand{\C}{\mathbb{C}}%
\newcommand{\1}{\mathbf{1}}%

\newcommand{\FF}{\mathcal{F}}%
\newcommand{\CC}{\mathcal{C}}%
\newcommand{\DD}{\mathcal{D}}%
\newcommand{\GG}{\mathcal{G}}%
\newcommand{\Sch}{\mathcal{S}}%
\newcommand{\ZZ}{\mathcal{Z}}%
\newcommand{\KK}{\mathcal{K}}%

\newcommand{\mellin}{{\mathcal{M}}}%
\newcommand{\fourier}{{\mathcal{F}}}%
\newcommand{\laplace}{{\mathcal{L}}}%

\newcommand{\ffi}{\varphi}%
\newcommand{\RE}{\operatorname{Re}}%
\begin{document}

\title
{Fractional It\^o Calculus for Randomly Scaled Fractional Brownian Motion and its Applications to  Evolution Equations}

\author{
Yana A. Butko\footnote{Kassel University, Institute of Mathematics, Heinrich-Plett-Str. 40, 34132 Kassel, Germany, \texttt{kinderknecht@mathematik.uni-kassel.de}}~ 
and Merten Mlinarzik\footnote{Technical University of Braunschweig, Institute for Analysis and Algebra, Universit\"atsplatz 2,
38106 Braunschweig, Germany, \texttt{m.mlinarzik@tu-braunschweig.de}}
}

\maketitle
\begin{abstract}
We define a  fractional It\^o stochastic integral with respect to a randomly scaled fractional Brownian motion via an $S$-transform approach. We investigate the properties of this stochastic integral, prove the It\^o formula for functions of such stochastic integrals and  apply this It\^o formula for investigation of related generalized time-fractional evolution  equations. 
\end{abstract}


\section{Introduction}
This  research is motivated by the \emphasis{superstatistical fractional Brownian motion (FBM) model} of anomalous diffusion~\cite{mura_etal-jpa-2008,molina_etal-pre-2016,mackala_etal-pre-2019,
itto_etal-jrsi-2021}.  
Under anomalous diffusion one understands diffusion type phenomena which occur qualitatively faster or slower than the classical diffusion and/or which are governed by non-Gaussian processes.
The experimental evidence of anomalous diffusion in living systems has been definitively established (see, e.g.,~\cite{dieterich_etal-pnas-2008,vanhaastert_etal-pnas-2009,cherstvy_etal-pccp-2018,matthaus_etal-bj-2009,wu_etal-pnas-2014,alves_etal-plosone-2016,souza_etal-plosone-2017,golding_etal-prl-2006, weiss_etal-bj-2004}).
In particular, anomalous diffusion has been reported  (see, e.g.,  \cite{barkai_etal-pt-2012,hofling_etal-rpp-2013,manzo_etal-rpp-2015}) by analysing  the celebrated dataset of measurements of the motion of individual fluorescently labeled mRNA molecules inside live E. coli cells by Golding \& Cox \cite{golding_etal-prl-2006}.
The superstatistical FBM model of anomalous diffusion considers this phenomenon as a result of inhomogeniety of the ensemble of diffusing particles (what is reasonable in the case of diffusion of complex macromolecules with different individual charachteristics) and inhomogeniety of the environment. This model suggests, as an underlying process for each single diffusing ``test-particle'', a process $X = \of{X_t}_{t \geq 0}$ of the form
	\begin{align}\label{eq:X}
		X_t \coloneqq \sqrt{A}B^H_t \comma \largespace t \geq 0
		\comma
	\end{align}
where $B^H = \of{B^H_t}_{t \geq 0}$ is an FBM with Hurst parameter $H \in \of{0,1}$ and $A$ is a positive random variable which is stochastically independent from $B^H$. 
The FBM $B^H$ is responsible for  the nonlinearity of the variance in time and arises due to inhomogeniety of the environment. Physical origin of FBM in this context has been  mathematically established in~\cite{BBDP}.
The random variable $A$ plays the role of a (random) diffusion coefficient of an inhomogeneous diffusing test-particle, and  is responsible for the non-Gaussianity of the process $X$.  
In this respect, we remind that the experimental evidence of a population of diffusion coefficients has been reported, e.g., in~\cite{mackala_etal-pre-2019,grimes_etal-c-2023,manzo_etal-prx-2015},
the FBM has been experimentally verified as  the underlying stochastic motion in many living systems, see, e.g.~\cite{magdziarz_etal-prl-2009,szymanski_etal-prl-2009,weiss-pre-2013}.   Note also that the superstatistical FBM is a self-similar process with stationary increments what makes such model  attractive for proceeding numerical simulations. 
\par
Superstatistical FBM models~\eqref{eq:X} belong to the family  of \emphasis{Randomly Scaled Gaussian Processes (RSGPs)} with $B^H$ as the Gaussian part and the random scaling $\sqrt{A}$. 
In the present paper, we call the processes of the form~\eqref{eq:X} \emphasis{Randomly Scaled FBM}. 
The processes of the form~\eqref{eq:X} with  different choices for the distribution of the random variable $A$ 
have been actively discussed and investigated in the literature nowadays, in particular in relation to evolution equations with time-fractional derivatives.
Historically, the first representative of this class of processes is \emphasis{Grey Brownian Motion (GBM)}, introduced by Schneider as a stochastic solution of time-fractional evolution equation~~\cite{MR1124240,MR1190506,zbMATH04027994,zbMATH04132860}. 
The Laplace transform of the corresponding random variable $A$ is given by the Mittag-Leffler function $E_\beta\of{-\cdot}$ with $\beta \coloneqq 2H$, $H \in \of{0,1/2}$.
Nowadays, the family of  randomly scaled  FBMs  includes the \emphasis{Generalised Grey Brownian Motion (GGBM)}~\cite{mura_etal-jpa-2008,MR2588003,MR2501791} (which provides a stochastic solution to an Erd\'elyi--Kober fractional-type differential equation~\cite{zbMATH06194179} and is largely stuided including, e.g., local time~\cite{zbMATH06467391}, structure factors~\cite{MR3956716}, stochastic currents~\cite{arXiv:2408.10936},
integral representation~\cite{zbMATH07553643} and the Feynman-Kac formula~\cite{BBB22}), processes governed by the space-time fractional diffusion equation~\cite{MR3513003} or by evolution equations containing Marichev--Saigo--Maeda time-fractional operators~\cite{BB22}. 
Highly general Fokker--Planck equations for a wide class of superstatiscal FBM models  have been recently derived in~\cite{zbMATH07893183}.  
Randomly scaled FBMs have been actively investigated in the framework of non-Gaussian counterparts of White Noise Analysis: the Grey Noise Analysis  by Schneider (for GBM)~\cite{MR1124240,MR1190506}, 
the Mittag--Leffler Analysis (for GGBM)~\cite{GJ15,GJ16} and the Gamma-Grey Analysis (for the so-called \emphasis{Gamma-Grey Brownian Motion})~\cite{BCG23}.  
Furthermore, randomly scaled FMBs whose random ``diffusion coefficient'' $A$  follows the  generalized Gamma distribution~\cite{Sposini_2018} (in particular, the Weibull distribution~\cite{PhysRevE.99.012143}), or the log-normal distribution~\cite{dossantos_etal-p-2020},  have been considered as suitable physical models of anomalous diffusion.
\par
One of the ways to obtain the relation between stochastic processes and evolution equations (in terms of Feynman-Kac type formulas) is given via Stochastic Calculus.
Since FBM is not a semimartingale, It\^o’s stochastic integration theory cannot be applied. 
Therefore, FBM has been treated in the frame of other types of integration theories (see the overview in Introduction of~\cite{B03a} for further information).
In the paper~\cite{B03a}, the integration theory for FBM has been developed in the frame of White Noise Analysis. 
Moreover, in the case of ``sufficiently good'' integrands, the integration theory can be formulated in a much simpler way by means of the so-called \emphasis{$S$-transform}~\cite{B03b}. 
And the resulting \emphasis{fractional It\^o stochastic integral} is essentially the same integral as its counterparts constructed in the frame of White Noise Analysis, or in the frame of Malliavin Calculus.
\par
In the present paper, we extend the $S$-transform approach of Bender~\cite{B03b} to the case of randomly scaled FBM. 
We suggest suitable extensions of the notions of Wick exponential and $S$-transform, define a fractional It\^o stochastic integral with respect to a randomly scaled FBM, investigate the properties of this stochastic integral, prove the It\^o formula for functions of such stochastic integrals (in the case of ``nice'' integrands). 
Finally, we apply this It\^o formula for investigation of related evolution  equations.  
In particular, we restore several known relations between evolution equations and randomly scaled FBMs (established in~\cite{BB22,BCG23,RVP22}) within our setting.

\section{Preliminaries}
Let $L^2\of{\R}$ denote the space of Borel-measurable and square integrable functions on $\R$ and denote the usual norm on $L^2\of{\R}$ by $\norm{\cdot}_{2}$.
We define two fractional operators. 
Let $H \in \of{0,1}$ and let the strictly positive constant $K_H$ be given by
	\begin{align*}
		K_H 
		\coloneqq 
		\Gamma\of{H + \frac{1}{2}} \of{\int_0^\infty \of{1 + s}^{H - 1/2} - s^{H - 1/2}\dd{s} + \frac{1}{2H}}^{-1/2} 
		\period
	\end{align*}	
We define the linear operators $\of{M^H_\pm,\DD\of{M^H_\pm}}$ acting on the set $\mathcal{M}\of{\R}$ of Borel-measurable functions through their actions. 
For $H \in \of{\frac{1}{2}, 1}$, we define the \emphasis{fractional integral of Weyl's type}
	\begin{align*}
		\of{M_{-}^{H} f}\of{x} 
		\coloneqq 
		\frac{K_H}{\Gamma \of{H - 1/2}} \int_{x}^{\infty} f\of{t} \of{t - x}^{H - 3/2} \dd{t}
	\end{align*}
and  
	\begin{align*}
		\of{M_{+}^{H} f}\of{x} 
		\coloneqq  
		\frac{K_H}{\Gamma \of{H - 1/2}} \int_{-\infty}^{x} f\of{t} \of{x - t}^{H - 3/2} \dd{t} 
		\period
	\end{align*}
In this case, the domain $\DD\of{M_{\pm}^{H}}$ is composed of all $f \in \mathcal{M}\of{\R}$ for which the respective integrals exist almost surely.
For $H \in \of{0, \frac{1}{2}}$, we define the \emphasis{fractional derivative of Marchaud's type}
	\begin{align*}
		M_\pm^H f 
		\coloneqq 
		\lim_{\epsilon \to 0} D^H_{\pm , \epsilon}f 
		\comma
	\end{align*} 
where
	\begin{align*}
		\of{D^H_{\pm , \epsilon}f}\of{x} 
		\coloneqq 
		\frac{K_H \of{1/2 - H} }{\Gamma \of{1/2 + H}} \int_{\epsilon}^{\infty} \frac{f\of{x} - f\of{x \mp t}}{t^{3/2 - H}} \dd{t} 
		\period
	\end{align*}
Here the domain $\DD\of{M^H_\pm}$ is composed of all $f \in \mathcal{M}\of{\R}$ for which the above limit exists almost surely. 
Lastly, for $H = 1/2$, we take $M^{1/2}_{\pm}$ to be the identity operator.
An important class of functions in $\DD\of{M^H_-}$ are the indicator functions of intervals  
	\begin{align*}
		\1_{\of{s,t}}\of{x} 
		\coloneqq
		\begin{cases}
			1  , & \text{if } s \leq x < t, \\
			-1 , & \text{if } t \leq x < s, \\
			0  , & \text{otherwise},
		\end{cases}
	\end{align*} 
with $\1_{\of{s,t}} \in L^2\of{\R}$ for $s$, $t \in \R$.
\begin{lemma} \label{result:MNorm}
	For $H \in \of{0,1}$ with $H \neq 1 / 2$ and $s, t \in \R$ the function $\1_{\of{s,t}}$ is a member of $\DD\of{M_-^H}$, and $M_-^H \1_{\of{s,t}}$ is given by 
		\begin{align*}
			\of{M_-^H \1_{\of{s,t}}}\of{x} 
			= 
			\frac{K_H}{\Gamma\of{H + 1/2}} \of{\of{t - x}^{H - 1/2}_+ - \of{s - x}^{H - 1/2}_+} 
			\period
		\end{align*}
		In particular, $M_-^H {\1}_{\of{s,t}}$ is an $L^2\of{\R}$ function with the norm 
		\begin{align*}
			\norm{M_-^H  {\1}_{\of{s,t}}}_{2} 
			= 
			\abs{t-s}^H 
			\period
		\end{align*}
\end{lemma}

\begin{proof}
	See Lemma~1.1.3 and Remark~1.1.4 in~\cite{M08}. 
	The $L^2$-norm of $M_{-}^{H}  {\1}_{\of{s,t}}$ can be calculated in a similar way as on page~364 of~\cite{M08}.
\end{proof}
\par
Recall that a \emphasis{fractional Brownian motion} $B^H = \of{B_t^H}_{t \in \R}$ with \emphasis{Hurst parameter} $H \in \of{0,1}$ is a centered Gaussian process with continuous paths and covariance structure 
	\begin{align*}
		\operatorname{Cov}\bra{B_t^H,B_s^H} 
		= 
		\frac{1}{2}\of{\abs{t}^{2H} + \abs{s}^{2H} - \abs{t - s}^{2H}} 
		\period
	\end{align*}
One can may always assume $B_0^H = 0$.
One can construct a fractional Brownian motion from a two-sided Brownian motion $B$ using the notion of the \emphasis{Wiener integral} of a function $f \in L^2\of{\R}$,
	\begin{align*}
		I\of{f} 
		\coloneqq 
		\int_\R f\of{t} \dd{B_t}
		\period
	\end{align*}
Indeed, a probability space  carrying a two-sided Brownian motion $B$ always carries a fractional Brownian motion,
given by a continuous version of the process $\of{I\of{M_-^H \1_{\of{0,t}}}}_{t \in \R}$, 
with arbitrary Hurst parameter $H \in \of{0,1}$. 
In particular $B^H$ is measurable, with respect to $B$. 
In the sequel, we use the concept of Pettis integration.
\par
Let $\mathcal{H}$ be a separable Hilbert space with scalar product $\sca{\cdot}{\cdot}$ and induced norm $\norm{\cdot}$. 
Furthermore, let $M \subset \R$ be a Borel set and let $f : M \to \mathcal{H}$ be a function. 
We call $f$ \emphasis{Pettis measurable} if the mappings
	\begin{align}\label{eq:mappings}
		M \to \R \comma \largespace t \mapsto \sca{f\of{t}}{\psi}
	\end{align}
are (Borel) measurable for all $\psi \in \mathcal{H}$.
Note, if there is  a dense subset $D$ of $\mathcal{H}$ such that the mappings~\eqref{eq:mappings} are measurable for all $\psi \in D$, then $f$ is already Pettis measurable.
For a Pettis measurable function $f : M \to \mathcal{H}$, 
we say that $f$ is \emphasis{Pettis integrable} if the mappings~\eqref{eq:mappings} are in $L^1\of{M}$ for all $\psi \in \mathcal{H}$. 
In this case, there exists a uniquely determined element $F \in \mathcal{H}$ with
	\begin{align*}
		\sca{F}{\psi} = 
		\int_M \sca{f\of{t}}{\psi} \dd{t} 
	\end{align*}
for all $\psi \in \mathcal{H}$. 
We call $F$ the \emphasis{Pettis integral} of $f$ and simply write
	\begin{align*}
		F 
		= 
		\int_M f\of{t} \dd{t} 
		\period
	\end{align*}
The existence and uniqueness of $F$ are a consequence of the Riesz representation theorem. 
An useful criterion for Pettis integrability of a function is given by integrability of its norm function.

\begin{proposition} \label{result:pettisExist}
	Let $M \subset \R$ be a Borel set and let $f : M \to \mathcal{H}$ be a Pettis measurable function. 
	Furthermore, let the norm mapping
		\begin{align*}
			M \to \R \comma \largespace t \mapsto \norm{f\of{t}}
		\end{align*}
		be in $L^1\of{M}$. 
		Then $f$ is Pettis integrable and we have the following norm estimate for the integral
		\begin{align*}
			\norm{\int_M f\of{t} \dd{t}} 
			\leq 
			\int _M \norm{f\of{t}} \dd{t} 
			\period
		\end{align*}
\end{proposition}

\begin{remark} \label{result:pettisInter}
	One can easily  see that the Pettis integral interchanges with the scalar product, 
	and we have for a Pettis integrable function $f : M \to \mathcal{H}$ and $\psi \in \mathcal{H}$ that
		\begin{align*}
			\int_M \sca{f\of{t}}{\phi} \dd{t} 
			= 
			\sca{\int_M f(t) \dd{t}}{\psi} 
			\period
		\end{align*}
\end{remark}
\par
Another concept needed for the following discussion is the Laplace transform of a probability measure on $\of{0,\infty}$.
Indeed let $\of{\of{0,\infty}, \mathcal{B}\of{0,\infty},\mu}$ be a probability space. 
We then denote the \emphasis{Laplace transform} of $\mu$ by 
	\begin{align*}
		\laplace\bra{\mu} \smallspace : \smallspace 
		\left[0,\infty\right) \to \R \comma \largespace 
		t \mapsto \int_0^{\infty} \e{-t x} \smallspace\dd[\mu]{\of{x}} 
		\period
	\end{align*} 
Accordingly, we understand the Laplace transform of a random variable $A$ with values in $\of{\of{0,\infty}, \mathfrak{B}\of{0,\infty}}$ as  the Laplace transform of its distribution $\Pb^A=\Pb\circ A^{-1}$,
	\begin{align*}
		\laplace\bra{A} 
		\coloneqq 
		\laplace\bra{\Pb^A} 
		\period
	\end{align*}
Obviously, $\laplace\bra{\mu}\of{t} \in \of{0,1}$ for $t \in \of{0,\infty}$,
 as well as $\laplace\bra{\mu}\of{0} = 1$. 
Furthermore, $\laplace\bra{\mu}$ is continuous on $[0,\infty)$ and infinitely often differentiable on $\of{0,\infty}$ with 
	\begin{align*}
		\diffQ[n]{\laplace\bra{\mu}}{t}\of{t}
		= 
		\int_0^{\infty} \of{-x}^n \e{-t x} \smallspace\dd[\mu]{\of{x}}
	\end{align*}
for all $t \in \of{0,\infty}$. 
Moreover, by the Bernstein theorem, a function $\ffi : \of{0,\infty} \to \C$ is the Laplace transform of a (unique) probability measure $\mu$ on $\of{\of{0,\infty},\mathcal{B}\of{0,\infty}}$ if and only if $\ffi\of{0} = 1$ and $\ffi$ is a completely monotone function, that is 
	\begin{align*}
		\of{-1}^n \diffQ[n]{\ffi}{t}\of{t} \in [0,\infty)
	\end{align*}
for all $t > 0$ and $n \in \N_0$. 
A completely monotone function $\ffi \coloneqq \laplace\bra{\mu}$ extends to a holomorphic function on $\set{z \in \C}{\RE z > 0}$. 
If, moreover, $\ffi$ extends to a holomorphic function on an open ball $B_R(0) \coloneqq \set{z \in \C}{\abs{z} < R}$, then the integral
	\begin{align*}
		\int_0^{\infty} \e{-z x} \dd[\mu]{\of{x}}
	\end{align*}
is finite for all $z$ with $\RE z \in(-R,0)$ and evaluates to $\ffi\of{z}$ on $B_R(0)$ by the Raikov theorem (cf. Theorem~3.1.2 in~\cite{L60}).

\section{Fractional It\^{o} Integral with respect to a Randomly Scaled Fractional Brownian Motion} \label{section:frac_integral}
We now consider a probability space $\of{\Omega \comma \FF \comma \Pb}$ carrying a two-sided Brownian motion $B$ together with an independent and strictly positive random variable $A$. 
Furthermore, we assume that there is a maximal (and possibly infinite) radius $R_A > 0$ such that the Laplace transform of $A$ is holomorphic on the open ball $B_{R_A}\of{0} \subset \mathbb{C}$. 
We denote by 
	\begin{align*}
		L^2_A \coloneqq L^2\of{\Omega , \sigma\of{B,A} , \Pb}
	\end{align*}
the space of square integrable random variables which are measurable with respect to $B$ and $A$, equipped with the usual norm $\norm{\cdot}_{L^2_A}$. Like in the previous section, let $I$ denote the Wiener integral with respect to $B$ and with values in $L^2_A$.

\begin{definition} \label{definition:WickExp}
	We define the set $D_A$ of admissible test-functions by
		\begin{align*}
			D_A 
			\coloneqq 
			\set{\eta \in \Sch\of{\R}}{2 \norm{\eta}^2_{2} < R_A}
			\subset 
			L^2\of{\R} 
			\period
		\end{align*}		
We then define the \emphasis{Wick-$A$ exponential} of a function $\eta \in D_A$ as 
		\begin{align*}
			W_A\of{\eta} 
			\coloneqq 
			\e{\sqrt{A} I\of{\eta}} 
			\productspacing
			{\E\bra{\e{\sqrt{A} I\of{\eta}}}}^{-1}
			\period
		\end{align*}
\end{definition}

For any $\eta \in D_A$, the normalizing factor in $W_A\of{\eta}$ is given by
	\begin{align*}
		\E\bra{\e{\sqrt{A} I\of{\eta}}} 
		&=
		\E\bra{  \E\cond{\e{\sqrt{A}I\of{\eta}}}{A}} 
		\\ &= 
		\E\bra{\e{\frac{A}{2} \norm{\eta}_{2}^2}} 
		=
		\laplace\bra{A}\of{-\frac{1}{2}\norm{\eta}_{2}^2} 
		\period 
	\end{align*}	 
Hence, for  $ \eta \in D_A$, we have
	\begin{align*}
		\norm{W_A\of{\eta}}_{L^2_A}^2 
		= 
		\laplace\bra{A}\of{-2\norm{\eta}_{2}^2}\laplace
		\productspacing
		\bra{A}\of{-\frac{1}{2}\norm{\eta}_{2}^2}^{-2}
	\end{align*}
and hence the mapping
	\begin{align*}
		W_A \smallspace : \smallspace 
		D_A \to L^2_A \comma \largespace 
		\eta \mapsto W_A\of{\eta}	
	\end{align*}
is a (non-linear) continuous mapping.
We will now use the class of Wick-$A$ exponentials as test functions on the space $L^2_A$.

\begin{definition}[$S_A$-transform]
	For $\Theta \in L^2_A$ we define the \emphasis{$S_A$-transform} of $\Theta$ as the \emphasis{non-linear} functional $S_A\bra{\Theta}$ on the set of admisible test-functions defined by 
	\begin{align*}
		S_A\bra{\Theta} \smallspace : \smallspace 
		D_A \to \R \comma \largespace 
		\eta \mapsto \E\bra{\Theta W_A\of{\eta}} 
		\period
	\end{align*}
\end{definition}

In other words, the $S_A$-transform of a random variable $\Theta$ tests $\Theta$ against the class of Wick-$A$ exponentials of admissible Schwartz functions. 
We will see in the next theorem that this test class is large enough to determine the tested functions uniquely in $L^2_A$. 

\begin{theorem} \label{result:SAInj}
	For $\Theta$ and $\Phi$ in $L^2_A$ we have that $S_A\bra{\Theta}(\eta) = S_A\bra{\Phi}(\eta)$ for all $\eta\in D_A$ if and only if $\Theta = \Phi$.
\end{theorem}

\begin{proof}
	As the $S_A$-transform is linear with respect to $\Theta$, 
	it suffices to show that $S_A\bra{\Theta}\of{\eta} = 0$ for all $\eta$ in  $D_A$ already implies $\Theta = 0$. 
	Hence, let $\Theta \in L^2_A$ be such that $S_A\bra{\Theta}(\eta)$ vanishes for all $\eta\in D_A$.
	Furthermore, let $\sset{\xi_n}_{n \in \N} \subset \mathcal{S}\of{\R}$ be an orthonormal basis of $L^2\of{\R}$, 
	and define an increasing sequence of $\sigma$-algebras $(\FF_n)_{n\in\N}$ by 
		\begin{align*}
			\mathcal{F}_n 
			\coloneqq 
			\sigma\of{\set{\sqrt{A} I\of{\xi_k}}{k \leq n}} \period
		\end{align*}	 
	Since the  $\sigma$-algebra $\sigma\of{A,B}$ of the space $L^2_A$ is generated by the random variables of type $\sqrt{A} I\of{\eta}$ with $\eta \in \mathcal{S}\of{\R}$, the sets $\FF_n$ increase to $\sigma\of{A,B}$.
	Further, consider the  conditional expected values of $\Theta$, denoted by 
		\begin{align*}
			\Theta_n 
			\coloneqq 
			\E\bra{\Theta \smallspace \vert \smallspace \FF_n} \in L^2_A
			\comma
		\end{align*}
	and show that $\Theta_n = 0$ for all $n \in \N$. 
	To that end, we fix an $n \in \N$ and choose a Borel-measurable function $\theta_n : \R^n \to \overline{\R}$ with 
	\begin{align*}
		\Theta_n 
		= 
		\theta_n\of{\sqrt{A}I\of{\xi_1}, \ldots, \sqrt{A}I\of{\xi_n}}
	\end{align*}
	via the factorisation lemma.
	As the $\xi_k$ are orthonormal, 
	we have that the Wiener integrals $I\of{\xi_k}$ are i.i.d. with standard normal distribution. 
	In other words, the random vector $\of{I\of{\xi_1} , \ldots , I\of{\xi_n}}$ has a standard ($n$-dimensional) normal distribution with Lebesgue density 
		\begin{align*}
			\rho_n\of{x} 
			\coloneqq 
			\of{2\pi}^{-n/2}\exp\of{-\frac{1}{2}\abs{x}^2}
			\period
		\end{align*}	
	Respectively, the randomly scaled random vector $\of{\sqrt{A}I\of{\xi_1}, \ldots, \sqrt{A}I\of{\xi_n}}$ has the Lebesgue density
		\begin{align*}
			\rho_n^{A}\of{x} 
		 	\coloneqq 
		 	\E\bra{A^{-n/2}\rho_n\of{A^{-1/2}x}}
		\end{align*}
	as the diffusion coefficient $A$ is independent from all $I\of{\xi_k}$.
	For $t = \of{t_1 , \ldots , t_n} \in \R^n$ with $2\abs{t}^2 < R_A$, 
	we now define $\eta_t$  as the linear combination 
		\begin{align*}
			\eta_t 
			\coloneqq 
			\sum_{k = 1}^n t_k \xi_k 
			\period
		\end{align*}	 
	Hence $\eta_t \in D_A$. And with the law of total expectation
		\begin{align*}
			0 
			&= 
			\E\bra{\e{\sqrt{A} I\of{\eta_t}}}
			\productspacing S_A\bra{\Theta}\of{\eta_t}
			= 
			\E\bra{\Theta \e{\sqrt{A} I\of{\eta_t}}} 
			\\ &= 
			\E\bra{\Theta_n \e{\sqrt{A} I\of{\eta_t}}} 
			=
			\int_{\R^n} \theta_n\of{x} \rho_n^{A}\of{x} \e{t \cdot x} \dd{x}
			\period
		\end{align*}
	Since the last integral in the above equation exists and vanishes for all $t$ with $2\abs{t}^2 < R_A$,  
	we can conclude that for such a fixed $t$ the mapping 
		\begin{align*}
			\set{z \in \C}{2 \abs{\RE z}^2 < R_A \abs{t}^{-2}} \to \mathbb{C} \comma \largespace 
			z \mapsto \int_{\R^n} \theta_n\of{x} \rho_n^{A}\of{x} \e{z t \cdot x} \dd{x}
		\end{align*}
	is holomorphic and vanishing. 
	As a consequence, the Fourier transform of the function $\theta_n \rho_n^{A}$ must vanish on $\R^n$. 
	The strict positivity of $\rho_n^{A}$ then already implies $\theta_n = 0$ and consequently $\Theta_n = 0$. 
	As $\FF_n$ increases to $\sigma\of{B, A}$, we also have that $\Theta = 0$.
	This concludes the proof.
\end{proof}

In the case when $\laplace[A] = E_\beta(-\cdot)$, 
where $E_\beta$ is the Mittag-Leffler function with parameter $\beta\in(0,1]$,
the $S_A$-transform is equivalent to $S_\beta$-transform introduced in~\cite{GJ15, GJ16}. 
In the case when the distribution of $A$ is the Dirac delta-measure concentrated at the point $1$, 
i.e. $A \equiv 1$, the $S_A$-transform coincides with the \emphasis{$S$-transform} as defined, e.g., in~\cite{B03b}. 
Moreover, in general, we can express the $S_A$-transform in terms of the $S$-transform. 
This expression will allow us to make use of the unscaled theory developed in~\cite{B03b} in the case of a non-trivial diffusion coefficient. 
For that we will denote by $B_{\bullet}$ the Brownian motion as a random element on $\mathcal{C}(\R)$, i.e. 
	\begin{align*}
		B_{\bullet} \smallspace : \smallspace 
		\Omega \to \mathcal{C}\of{\R} \comma \largespace 
		\omega \mapsto B_{\cdot}\of{\omega} \period
	\end{align*}

\begin{theorem} \label{result:SAtoS}
	Let $\Theta \in L^2_A$ and let $\theta : \R \times \mathcal{C}\of{\R} \to \overline{\R}$ be a Borel-measurable function such that $\Theta = \theta\of{A,B_{\bullet}}$.
	Then the $S_A$ transform of $\Theta$ can be represented as
		\begin{align*}
			S_A\bra{\Theta}\of{\eta} 
			= 
			\E\bra{\of{\vphantom{\Big\vert} S\bra{\theta\of{a, B_{\bullet}}}\of{\sqrt{a} \eta}}_{a = A}
			\productspacing \varepsilon_A^\eta\of{A}}
		\end{align*}
	for all $\eta \in D_A$. 
	The function $\varepsilon_A^\eta : \of{0,\infty} \to \R$ is given by
		\begin{align*}
			\varepsilon_A^\eta\of{a} 
			\coloneqq 
			\e{\frac{a}{2} \norm{\eta}_{2}^2}
			\productspacing
			\E\bra{\e{\sqrt{A} I\of{\eta}}}^{-1}
			\period
	\end{align*}
\end{theorem}

\begin{proof}
	Notice that $\theta\of{a,B_{\bullet}}$ is $\sigma\of{B}$-measurable for all $a > 0$, and it holds
		\begin{align*}	\E\bra{\of{\norm{\theta\of{a,B_{\bullet}}}_{L^2\of{\Omega , \sigma(B) , \Pb}}^2}_{a = A}} 
			= 
			\norm{\Theta}_{L^2_A}^2 
			< 
			\infty 
			\period
		\end{align*}
	Thus, $\theta\of{a,B_{\bullet}} \in L^2\of{\Omega , \sigma\of{B} , \Pb}$ for $\Pb^A$-almost all $a > 0$, and we have
		\begin{align*}
			S_A\bra{\Theta}\of{\eta} 
			&= 
			\E\bra{ \vphantom{\Big\vert} \smallspace \E\cond{\Theta W_A\of{\eta}}{A}} 
			\\ &=
			\E\bra{ \vphantom{\Big\vert} \smallspace \E\bra{ \vphantom{\Big\vert}\theta\of{a,B_{\bullet}} \e{\sqrt{a}I\of{\eta}}}_{a = A}}
			\productspacing
			\E\bra{\e{\sqrt{A} I\of{\eta}}}^{-1}
			\\ &=	
			\E\bra{\of{ \vphantom{\Big\vert} S\bra{\theta\of{a,B_\bullet}} \e{\frac{a}{2}\norm{\eta}_{2}^2}}_{a = A}}
			\productspacing
			\E\bra{\e{\sqrt{A} I\of{\eta}}}^{-1}
			\\ &=	
			\E\bra{\of{ \vphantom{\Big\vert} S\bra{\theta\of{a , B_{\bullet}}}\of{\sqrt{a} \eta}}_{a = A}
			\productspacing
			\varepsilon_A^\eta\of{A}}
			\period
		\end{align*}
\end{proof}

We are now able to define a generalization of the fractional It\^{o} integral with respect to fractional Brownian motion as in~\cite{B03b} to an integral with respect to a randomly scaled fractional Brownian motion.
For that we construct a fractional Brownian motion $B^H$ with an arbitrary Hurst parameter $H \in \of{0,1}$ on $\of{\Omega , \FF , \Pb}$ as described in the preliminaries. We will then denote the \emphasis{randomly scaled fractional Brownian motion} resulting from scaling $B^H$ with $\sqrt{A}$ by 
	\begin{align*}
		X 
		\coloneqq 
		X^{A,H} \coloneqq \sqrt{A}B^H 
		\period
	\end{align*}
As there is no danger of confusion, we will be omitting the superscripts of $X$ in the following.

\begin{remark} 
	The randomly scaled fractional Brownian motion $X$ is a mean-zero process with continuous paths and covariance structure
		\begin{align*}
			\operatorname{Cov}\bra{X_t, X_s} 
			= 
			\E\bra{A} \operatorname{Cov}\bra{B_t^H,B_s^H} 
			= 
			\frac{1}{2} \E\bra{A} \of{\abs{t}^{2H} + \abs{s}^{2H} - \abs{t - s}^{2H}} 
			\period
		\end{align*}
	Notice however that, in general, $X$ is not a Gaussian process. 
\end{remark}

We may now give the central definition of this paper.

\begin{definition}[Fractional Integral (\textit{1\textsuperscript{st} Approach})] \label{definition:fracInt}
	Let $M \subset \R$ be a Borel set and let $Z : M \to L^2_A$ be a process. 
	Furthermore, suppose that 
		\begin{align*}
			S_A\bra{A Z_t}\of{\eta} \of{M_+^{H}\eta}\of{t} \in L^1\of{M}
		\end{align*}
	with respect to $t$ for all $\eta \in D_A$. 
	We then call $Z$ \emphasis{(fractionally) integrable} with respect to $X$ if and only if there is an element $\Theta \in L^2_A$ such that
		\begin{align*}
			S_A\bra{\Theta}\of{\eta} 
			= 
			\int_M S_A\bra{A Z_t}\of{\eta} 
			\of{M_+^{H}\eta}\of{t} \dd{t}
		\end{align*}
	holds for all $\eta \in D_A$. 
	By the injectivity of the $S_A$-transform (Theorem~\ref{result:SAInj}), $\Theta$ is uniquely determined (and independent of the specific choice of $R_A$) should it exist. 
	In that case, we call $\Theta$ the \emphasis{fractional It\^{o}   integral} of $Z$ and write
		\begin{align*}
			\Theta 
			= 
			\int_M Z_t \dd{X_t} 
			\period
	\end{align*}
\end{definition}

\begin{theorem}
	Let $M \subset \R$ be a Borel set and let $Z,U : M \to L^2_A$ both be integrable with respect to $X$. 
	The fractional It\^{o} integral has the following properties :
		\begin{itemize}
			\item[$\mathrm{(i)}$] 
			The integral is linear, i.e. for $\alpha, \beta \in \R$ we have 
				\begin{align*}
					\int_M (\alpha Z_t + \beta U_t ) \dd{X_t} 
					= 
					\alpha \int_M Z_t \dd{X_t} + \beta \int_M U_t \dd{X_t} 
					\period
				\end{align*}	
			\item[$\mathrm{(ii)}$] 
			Let $M^{\prime} \subset M$ be a Borel subset of $M$. Then 
				\begin{align*}
					\int_{M^{\prime}} Z_t \dd{X_t} 
					= 
					\int_M Z_t \1_{M^{\prime}}\of{t} \dd{X_t}  
					\comma
				\end{align*}	
			whenever at least one of the integrals exists.
		 	\item[$\mathrm{(iii)}$] 
		 	The integral has zero expectation, i.e.
				\begin{align*}
					\E\bra{\int_M Z_t \dd{X_t}} 
					=
					0 
					\period
				\end{align*}
	\end{itemize}
\end{theorem}

\begin{proof}
	The linearity of the integral immediately follows from the definition, 
	since both the $S_A$-transform and the ordinary Lebesgue integral are linear. 
	Regarding the second statement we see that 
		\begin{align*}
			S_A\bra{\int_M Z_t \1_{M^{\prime}}\of{t} \dd{X_t}}\of{\eta} 
			&= 
			\int_M S_A\bra{A Z_t \1_{M^{\prime}}\of{t}}\of{\eta}  
			\of{M_+^{H}\eta}\of{t} \dd{t} 
			\\ &= 
			\int_M S_A\bra{A  Z_t}\of{\eta} 		
			\of{M_+^{H}\eta}\of{t}
			\productspacing 
			\1_{M^{\prime}}\of{t} 	\dd{t} 
			\\ &=
			\int_{M^{\prime}} S_A\bra{A Z_t}\of{\eta} \of{M_+^{H}\eta}\of{t} \dd{t}	
			=
			S_A\bra{\int_{M^\prime} Z_t \dd{X_t}}\of{\eta}
		\end{align*}
	holds for all $\eta \in D_A$ as long as one of the fractional It\^{o} integrals exists.
	Thus, by definition,
		\begin{align*}
			\int_{M^{\prime}} Z_t \dd{X_t} 
			= 
			\int_M Z_t \1_{M^{\prime}}\of{t} \dd{X_t} 
		\end{align*}
	if at least one fractional It\^{o} integral exists.
	Finally, the third statement follows from the fact that the expected value of the integral coincides with the $S_A$-transform of the integral evaluated at $\eta = 0$.
\end{proof}

For the trivial diffusion coefficient $A = 1$ this construction coincides with the fractional integral with respect to $B^H$ as constructed in~\cite{B03b}.
Under suitable conditions, we can express the integral with respect to $X$ in terms of the integral with respect to $B^H$.

\begin{theorem} \label{result:integralXtoBH}
	Let $M \subset \R$ be a Borel set and let $Z : M \to L^2_A$ be a process. 
	For all $t \in M$ let $z_t : \R \times \CC\of{\R} \to \overline{\R}$ be a (Borel-)measurable function such that $Z_t = z_t\of{A,B_{\bullet}}$. 
	If the process $\left(z_t\of{a,B_{\bullet}}\right)_{t\in M}$ is fractionally integrable with respect to $B^H$ for $\Pb^A$-almost all $a > 0$ then $Z$ is integrable with respect to $X$ with integral
		\begin{align*}
			\int_M Z_t \dd{X_t} 
			= 
			\of{\sqrt{a}\int_M z_t\of{a,B_{\bullet}} \dd {B^H_t}}_{a = A}
	\end{align*}
	if the right hand side is in $L^2_A$.
\end{theorem}

\begin{proof}
	As the right hand side is assumed to be in $L^2_A$ it is possible to calculate its $S_A$ transform using the definition of the  integral with respect to $B^H$ together with Theorem~\ref{result:SAtoS}.
	Indeed for all $\eta \in D_A$ it is
		\begin{align*}
			S_A\Big[ \Big( \sqrt{a} &\int_M z_t\of{a,B_{\bullet}} \dd{B^H_t} \Big)_{a = A} \Big] \of{\eta}
			\\ &= 
			\E\bra{\of{S\bra{\sqrt{a}\int_M z_t\of{a,B_{\bullet}} \dd{B^H_t}}\of{\sqrt{a}\eta}}_{a = A}
			\productspacing
			\varepsilon^\eta_A\of{A} }
			\\ &= 
			\E\bra{\of{\int_M S\bra{\sqrt{a} z_t\of{a,B_{\bullet}}}\of{\sqrt{a}\eta}
			\productspacing
			\sqrt{a}\of{M^H_+\eta}\of{t} \dd{t}}_{a = A}
			\productspacing
			\varepsilon^\eta_A\of{A}} 
			\\ &= 
			\E\bra{\int_M \of{ \vphantom{\Big\vert} S\bra{a z_t\of{a,B_{\bullet}}}\of{\sqrt{a}\eta}}_{a = A}
			\productspacing
			\of{M^H_+\eta}\of{t} \dd{t}
			\productspacing
			\varepsilon^\eta_A\of{A}} 
			\period
		\end{align*}
		Then, by interchanging the expectation and the integral with respect to $t$ in the last expression above, we obtain
			\begin{align*}
							S_A\Big[ \Big( \sqrt{a} &\int_M z_t\of{a,B_{\bullet}} \dd{B^H_t} \Big)_{a = A} \Big] \of{\eta}
			\\ &= 
			\int_M \E\bra{\of{ \vphantom{\Big\vert} S\bra{a z_t\of{a,B_{\bullet}}}\of{\sqrt{a}\eta}}_{a = A} 
			\productspacing
			\varepsilon^\eta_A\of{A}} \of{M^H_+\eta}\of{t} \dd{t} 
			\\ &= 
			\int_M   S_A\bra{A Z_t}\of{\eta} \of{M^H_+\eta}\of{t} \dd{t} 
			\period
			\end{align*}
	This immediately shows the statement by the definition of the randomly scaled fractional integral.
\end{proof}

This relation between the scaled and unscaled integrals motivates the following alternative  definition.

\begin{definition}[Fractional Integral (\textit{2\textsuperscript{nd} Approach})]
	Let $M \subset \R$ be a Borel set and let $Z : M \to L^2_A$ be given by $Z_t = z_t\of{A,B_{\bullet}}$ such that the processes $\left(z_t\of{a,B_{\bullet}}\right)_{t\in M}$ are fractionally integrable with respect to $B^H$ for $\Pb^A$-almost all $a > 0$. 
	We then call $Z$ fractionally integrable and define the integral as
		\begin{align*}
			\int_M Z_t \dd{X_t} 
			=
			\of{\sqrt{a}\int_M z_t\of{a,B_{\bullet}} \dd{B^H_t}}_{a = A} 
			\comma
		\end{align*}
	whenever the right hand side is in $L^2_A$.
\end{definition}

Theorem~\ref{result:integralXtoBH} tells us that the integral of the first approach is an extension of the integral of the second approach. 
The following corollary to the theorem presents a particularly simple case.

\begin{corollary} \label{result:Cor_integralXtoBH}
	Let $M \subset \R$ be a Borel set and let $Z : M \to L^2\of{\Omega , \sigma\of{B} , \Pb}$ be a process which is fractionally integrable with respect to $B^H$. 
	Then $Z$ is integrable with respect to $X$ in the sense of both approaches and we have 
		\begin{align*}
			\int_M Z_t \dd{X_t} 
			= 
			\sqrt{A} \int_M Z_t \dd{B^H_t} 
			\period
		\end{align*}
\end{corollary}

\section{Integral Processes} \label{section:integral_processes}
We remain in the setting of the previous section \ref{section:frac_integral}. 
In particular, $X = \sqrt{A}B^H$ still denotes the randomly scaled fractional Brownian motion with Hurst parameter $H \in \of{0, 1}$ and scaling $\sqrt{A}$.
\par
By integrating a given process $Z : \left[0, T\right) \to L^2_A$ with respect to $X$ over the sets $[0,t)$ with $t \leq T$ one may obtain a new process,
also indexed by the set $\left[0 , T \right)$,
provided the fractional integrals exist for all times $t$. 
In this section some regularity properties of certain process of this type are investigated.
To that end we first fix some notation.
\par
For $0 \leq T \leq \infty$ let $\overline{[0,T)}$ denote $[0,T]$ if $0\leq T<\infty$ and $[0,\infty)$ if $T=\infty$. 
In any case, consider the set
	\begin{align*}
		\mathfrak{DI}_T\of{X} 
		\coloneqq 
		\set{\overline{[0,T)} \to L^2_A \comma \mediumspace  t \mapsto \int_0^t \nu\of{s} \dd{X_s}}{
	\begin{matrix}
		&\nu \in \CC[0,\infty) \textnormal{ for } H \geq \frac{1}{2} \\
		&\nu \in \R \textnormal{ for } H < \frac{1}{2} 
	\end{matrix}}
	\end{align*}
of $L^2_A$-valued processes on $\overline{[0,T)}$ resulting from integrals of ``good'' deterministic functions with respect to $X$. 
The processes in $\mathfrak{DI}_T\of{X}$ exist by Corollary~\ref{result:Cor_integralXtoBH} above, Theorem~3.4 and Proposition~5.1 in~\cite{B03b}. 
Additionally, they show some regularity.
More precisely, we have the following two lemmas.

\begin{lemma} \label{result:intCont}
	Let $0 \leq T \leq \infty$. 
	Then every process in $\mathfrak{DI}_T\of{X}$ is $L^2_A$ continuous. 
\end{lemma}

\begin{proof}
	Let $Z \in \mathfrak{DI}_T\of{X}$ result from integrating $\nu$.	
	In the case of $H < 1/2$, meaning a constant $\nu \in \R$, it follows from Corollary~\ref{result:Cor_integralXtoBH} above and Example~3.1 in~\cite{B03b} that 	$Z_t = \nu X_t$. 
	The $L^2_A$ continuity then follows from the continuous covariance structure of $X$. 
	Now consider the case $H \geq 1 / 2$. 
	Again by Corollary~\ref{result:Cor_integralXtoBH} we can write $Z_t = \sqrt{A}Y_t$, where
		\begin{align}\label{eq:Y}
				Y_t 
				= 
				\int_0^t \nu\of{s} \dd{B^H_s} 
				\period
		\end{align}
	As the $L^2_A$ distance between $Z_t$ and $Z_s$ is given by 
		\begin{align*}
			\norm{Z_t - Z_s}_{L^2_A}^2 
			= 
			\E\bra{A} 
			\productspacing
			\norm{Y_t - Y_s}_{L^2_A}^2
		\end{align*}
	for all $t,s \in \overline{[0,T)}$, 
	it suffices to show that the process $Y$ is $L_A^2$ continuous. 
	To that end, consider $t \in \overline{[0,T)}$ and let $\of{t_n}_{n \in \N}$ be a sequence in $\overline{[0,T)}$ with $t_n \to t$. 
	According to Proposition~5.1 in~\cite{B03b}, 
	the difference $Y_t - Y_{t_n}$ is normal distributed with variance 
		\begin{align*}
			\norm{Y_t - Y_{t_n}}_{L^2_A}^2 
			= 
			\norm{M_-^H\of{\nu \1_{\of{t_n,t}}}}_{2}^2 
			\period
		\end{align*}
	Thus, we have by Lemma~\ref{result:MNorm}
		\begin{align*}
			\norm{Y_t - Y_{t_n}}_{L^2_A}^2
			&= 
			\norm{M_-^H\of{\nu\1_{\of{t_n,t}}}}_{2}^2
			\\ &\leq
			\of{\max_{s \in \bra{0,T}} \abs{\nu\of{s}}}^2 \norm{M_-^H \1_{\of{t_n,t}}}_{2}^2 \too{n \to \infty} 0 
		 \period
	\end{align*}
\end{proof}

\begin{lemma} \label{result:regularityVar}
	Let $0 \leq T \leq \infty$ and let $Z \in \mathfrak{DI}_T\of{X}$ result from integrating $\nu$. 
	Then the mapping
		\begin{align*}
			\of{0,T} \to \R \comma \largespace 
			t \mapsto \mathbb{E}\bra{Z^2_t} 
			= 
			\E\bra{A} 
			\productspacing
			\norm{M_-^H\of{\nu\1_{\of{0,t}}}}_{2}^2  
		\end{align*}
	is continuously differentiable on $\of{0,T}$ with the bound 
		\begin{align*}
			\abs{ \delimiterspacing \diff{t} \norm{M_-^H\of{\nu\1_{\of{0,t}}}}_{2}^2 \delimiterspacing } 
			\leq 
			2H \of{\max_{s \in \bra{0,T}}{\abs{\nu\of{s}}^2}} t^{2H - 1}
		\end{align*}
	for all $t \in \of{0,T}$.
\end{lemma}

\begin{proof}
	Recall that $\mathbb{E}\bra{Z^2_t}$ is indeed given by 
		\begin{align*}
			\mathbb{E}\bra{Z^2_t} 
			=
			\E\bra{A} 
			\productspacing
			\norm{M_-^H\of{\nu\1_{\of{0,t}}}}_{L^2\of{\R}}^2
		\end{align*}			
	according to Proposition~5.1 in~\cite{B03b}.
	In the case of $H = 1/2$, the statement is trivial due to $M^{1/2}_-$ being the identity operator.
	For $H < 1/2$ the statement easily follows from Lemma~\ref{result:MNorm} because then $\nu$ is constant. 
	Consequently, we now consider $H > 1/2$. 
	It has already been discussed in Lemma~5.2 in~\cite{B03b} that, in this case, $\norm{M_-^H\of{\nu\1_{\of{0,t}}}}_{2}^2$ is differentiable on $\of{0,T}$ with derivative 
		\begin{align}\label{eq:new}
			\diff{t}\norm{M_-^H\of{\nu\1_{\of{0,t}}}}_{2}^2 
			= 
			2H\of{2H-1}\int_0^t \nu\of{s}\nu\of{t}\of{t - s}^{2H-2} \dd{s} \period
		\end{align}
	The desired bound on the derivative follows easily from this expression. 
	Thus, it only remains to show that the derivative is continuous. 
	For that, it suffices to note that the mapping 
		\begin{align*}
			\of{0,T} \to \R \comma \largespace 
			t \mapsto \int_0^t \nu\of{s}\of{t - s}^{2H-2} \dd{s}	
		\end{align*}
	is continuous. 
\end{proof}

To extend the above results to functionals of processes in $\mathfrak{DI}_T\of{X}$,
we first specify which functionals we want to consider. 
For the following definition recall that the radius $R_A$ was defined to be the radius of the maximal open ball $B_{R_A}\of{0} \subset \C$ on which the Laplace transform of $A$ is holomorphic.

\begin{definition} \label{definition:admFunc}
	Let $T > 0$ and let $Z \in \mathfrak{DI}_T\of{X}$ result from integrating $\nu$.
	We call a function $F\of{t,z,a} \in \CC^{1,2,0}\of{\bra{0,T},\R, \of{0,\infty}}$ an \emphasis{admissible functional} of $Z$, if $F$ satisfies the growth bounds
		\begin{align}\label{eq:admFunc}
			\abs{F\of{t,\sqrt{a}z,a}} 
			\leq 
			C \e{ca + \lambda z^2} 
		\end{align}
	for some constants $C$, $c$, $\lambda \geq 0$ with $c < R_A / 2$ and $\lambda < \of{2T^H \max_{s \in \bra{0,T}}{\abs{\nu\of{s}}}}^{-2}$.
\end{definition}

We will show the $L^2_A$ continuity of the process $F\of{t,Z_t,A}$ for all admissible functionals $F$ of $Z$.
The first step towards this is to give an estimate of the $L_A^2$ norm of $F\of{s,Z_t,A}$.
	
\begin{lemma}	\label{result:GBound}
	Let $T > 0$ and let $Z \in \mathfrak{DI}_T\of{X}$ result from integrating $\nu$.
	Then for any admissible functional $F$ and $s,t \in \bra{0,T}$ we have $F\of{s,Z_t,A} \in L^2_A$. 
	Furthermore, the $L^2_A$ norm can be bounded independent of $s$ and $t$ by
		\begin{align*}
			\norm{F\of{s,Z_t,A}}_{L^2_A}^2 
			\leq 
			\frac{C^2 \laplace\bra{A}\of{-2c}}{\sqrt{1 - \of{2T^{H} \max_{s \in \bra{0,T}}{\abs{\nu\of{s}}}}^2\lambda}}
		\end{align*}
	for any growth constants $C$, $c$ and $\lambda$ of $F$ satisfying Definition~\ref{definition:admFunc}.
\end{lemma}		
		
\begin{proof}
	Naturally, $F\of{s,Z_t,A}$ is $\sigma\of{A,B}$-measurable for all $s,t \in \bra{0,T}$. 	
	Again we write $Z_t$ as $\sqrt{A}Y_t$, where $Y_t$ is given by~\eqref{eq:Y}.
	For any growth constants $C$, $c$ and $\lambda$ of $F$ satisfying Definition~\ref{definition:admFunc} we have
		\begin{align*}
			\norm{F\of{s,Z_t,A}}_{L^2_A}^2 
			&= 
			\E\bra{F\of{s,\sqrt{A}Y_t,A}^2} 
			\leq
			\E\bra{C^2\e{2\of{cA + \lambda Y_t^2}}} 
			\\ &=
			C^2 
			\productspacing			
			\E\bra{\e{2cA}}
			\productspacing 
			\E\bra{\e{2\lambda Y_t^2}} 
			=
			C^2 \laplace\bra{A}\of{-2c}
			\productspacing
			\E\bra{\e{2\lambda Y_t^2}}
			\period
		\end{align*}
	As $Y_t$ has a normal distribution with $\mathbb{E}\bra{Y^2_t} = \norm{M^H_-\of{\nu\1_{\of{0,t}}}}_{2}^2$ it is
		\begin{align*}
			\E\bra{\e{2\lambda Y_t^2}}  
			= 
			\of{1 - 4\lambda\mathbb{E}\bra{Y^2_t}}^{-1 / 2} 
			\leq
			\bigg( 1 - \of{2T^{H} \max_{s \in \bra{0,T}} \abs{\nu\of{s}}}^2\lambda \bigg)^{-\frac{1}{2}} 
			< 
			\infty 
			\period
	\end{align*}
	In the last estimate we used the assumed bound on $\lambda$ as well as Lemma~\ref{result:MNorm} via
		\begin{align*}
			\norm{M_-^H\of{\nu\1_{\of{0,t}}}}_{2}^2 
			\leq
			\norm{M_-^H\1_{\of{0,t}}}_{2}^2 \of{\max_{s \in \bra{0,T}} \abs{\nu\of{s}}}^2 \leq 
			T^{2H}\of{\max_{s \in \bra{0,T}} \abs{\nu\of{s}}}^2 
			\period
	\end{align*}
	This shows the desired estimate.
\end{proof}	
	
\begin{remark} \label{result:SofGBound}
	From the previous proof one can see in particular, 
	that for every $a \in \of{0,\infty}$ the $L^2\of{\Omega , \sigma\of{B} ,\Pb}$ norm of $F\of{t,\sqrt{a}Y_t,a}$ is bounded by 
		\begin{align*}
			\norm{F\of{t,\sqrt{a}Y_t,a}}_{L^2\of{\Omega, \sigma\of{B},\Pb}}^2 
			\leq
			\frac{C^2 \e{2ca}}{\sqrt{1 - \of{2T^{H} \max_{s \in \bra{0,T}}{\abs{\nu\of{s}}}}^2\lambda}} 
			\period
		\end{align*}
	This bound immediately implies 
		\begin{align*}
			\abs{S\bra{F\of{t,\sqrt{a}Y_t,a}}\of{\sqrt{a}\eta}} &\leq
			\of{\frac{C^2 \e{2ca}}{\sqrt{1 - \of{2T^{H} \max_{s \in \bra{0,T}}{\abs{\nu\of{s}}}}^2\lambda}}}^{1 / 2} \norm{W_1\of{\sqrt{a}\eta}}_{L^2} 	
			\\ &= 
			\frac{C \e{a \of{c + \frac{1}{2}\norm{\eta}^2_{2}}}}{\sqrt[4]{1 - \of{2T^{H} \max_{s \in \bra{0,T}}{\abs{\nu\of{s}}}}^2\lambda}}
	\end{align*}
	for all $\eta \in  D_A$.
	Since we have $0 < c + \norm{\eta}^2_{L^2(\R)} < R_A$ for all $\eta \in D_A$, it is
		\begin{align*}
			S\bra{F\of{t,\sqrt{a}Y_t,a}}\of{\sqrt{a}\eta} 
			\productspacing
			\varepsilon^\eta_A\of{a} 
			\in 
			L^1\of{\of{0,\infty} , \Pb^A}, 
		\end{align*}
	with $\varepsilon^\eta_A\of{a}$ as defined in Theorem~\ref{result:SAtoS}. 
\end{remark}		
	
Using the $L^2_A$ (and hence the stochastic) continuity of $Z \in \mathfrak{DI}_T\of{X}$, 
we can now show the $L^2_A$ continuity of $F\of{t,Z_t,A}$ for all admissible functionals $F$ of $Z$.

\begin{lemma} \label{result:funcCont}
	Let $T > 0$ and let $Z \in \mathfrak{DI}_T\of{X}$. For every admissible functional $F$ the process
		\begin{align*}
			\bra{0,T} \to L^2_A \comma \largespace 
			t \mapsto F\of{t,Z_t,A}
		\end{align*}
	is $L^2_A$ continuous.
\end{lemma}

\begin{proof}
	Let $t \in \bra{0,T}$ and let $\of{t_n}_{n \in \N}$ be a convergent sequence in $\bra{0,T}$ with $t_n \too{n \to \infty} t$.
	We will show
		\begin{align*}
			F\of{t_n,Z_{t_n},A} \too{L^2_A} F\of{t,Z_t,A}
		\end{align*}
	using the convergence theorem of Vitali.
	From Lemma~\ref{result:intCont} we know that $Z$ is $L^2_A$ continuous and hence continuous in probability, i.e. $Z_{t_n} \too{\Pb} Z_t$. 
	Thus, the continuous mapping theorem gives
		\begin{align*}
			F\of{t_n,Z_{t_n},A} \too{\Pb} F\of{t,Z_t,A} \period
		\end{align*}
	It remains to show that the sequence $\of{F\of{t_n,Z_{t_n},A}^2}_{n \in \N}$ is uniformly integrable.
	For that it suffices to find an $\epsilon > 0$ such that
		\begin{align*}
			\sup_{n \in \N} \smallspace \E\bra{\abs{F\of{t_n,Z_{t_n},A}}^{2 + \epsilon}}
			 < 
			\infty
			\period
		\end{align*}
	To do that, we take $C$, $c$ and $\lambda$ to be growth constants of $F$, satisfying Definition~\ref{definition:admFunc}, 
	and choose $\epsilon > 0$ small enough such that
		\begin{align*}
			\of{1 + \frac{\epsilon}{2}}\lambda < \of{2T^H \max_{s \in \bra{0,T}}{\abs{\nu\of{s}}}}^{-2}
			\largespace \textnormal{and} \largespace \of{2 + \epsilon}c < R_A
 		\end{align*}
	holds. 
	With this choice of $\epsilon$ we have for all $n \in \N$, similarly to the proof of Lemma~\ref{result:GBound},
		\begin{align*}
			 \E\bra{\abs{F\of{t_n,Z_{t_n},A}}^{2 + \epsilon}} 
			 &\leq
		 	C^{2+\epsilon} 
		 	\productspacing
		 	\E\bra{\e{\of{2+\epsilon}cA}}
		 	\productspacing
		 	\E\bra{\e{\of{2+\epsilon}\lambda Y_{t_n}^2}}
		 	 \\ &\leq
		 	\frac{C^{2 + \epsilon} \laplace\bra{A}\of{-\of{2+\epsilon}c}}{\sqrt{1 - \of{1 + \frac{\epsilon}{2}} \of{2T^{H} \max_{s \in \bra{0,T}}{\abs{\nu\of{s}}}}^2\lambda}} 
		 	< 
		 	\infty 
		 	\period
		\end{align*}
\end{proof}	

\begin{remark} \label{result:GCont}
	From the proof of Lemma~\ref{result:funcCont} it is clear that the process
		\begin{align*}
			\bra{0,T} \to L^2\of{\Omega, \sigma(B), \Pb} \comma \largespace 
			t \mapsto F\of{t,\sqrt{a}Y_t,a}
		\end{align*}
	is $L^2\of{\Omega, \sigma(B), \Pb}$ continuous for all $a \in \of{0,\infty}$.
\end{remark}

\section{The It\^{o} Formula}
The It\^o~formula derived in~\cite{B03b} can be generalized to the randomly scaled case with a non obvious dependency on $A$ in the term containing the second derivative of the functional. 
The It\^o formula will be concerned with admissible functionals $F\of{t,Z_t,A}$ of processes $Z \in \mathfrak{DI}_T\of{X}$. To formulate and proof the formula we again remain in the setting of section \ref{section:frac_integral} and \ref{section:integral_processes}.

\begin{theorem}[\textbf{It\^o formula}] \label{result:itoFormula}
	Let $T > 0$ and let $Z \in \mathfrak{DI}_T\of{X}$ result from integrating the function $\nu$, that is
		\begin{align*}
			Z \smallspace : \smallspace
			\bra{0,T} \to L^2_A \comma \largespace 
			t \mapsto \int_0^t \nu\of{s} \dd{X_s} \period
		\end{align*}
	Furthermore, let $F\of{t,z,a} \in \CC^{1,2,0}\of{\bra{0,T},\R, \of{0,\infty}}$ be such that $F$, $\pardiffQ{F}{t}$, $a \pardiffQ{F}{z}$ and $a \pardiffQ[2]{F}{z}$ are admissible functionals of $Z$.
	Then the following It\^o~formula holds true
		\begin{align} \label{eq:Ito}
			F\of{T,Z_T,A} 
			= 	 
			F\of{0,0,A} 
			&+ 
			\int_{0}^T \pardiffQ{F}{t} \of{t,Z_t,A} \dd{t}
			+ 
			\int_0^T\nu\of{t} \pardiffQ{F}{z} \of{t,Z_t,A} \dd{X_t} 
			\nonumber \\ &+ 
			\frac{A}{2} \int_0^T \diff{t} \norm{M_{-}^H\of{\nu\1_{\of{0,t}}}}^2_{2} 
			\productspacing 
			\pardiffQ[2]{F}{z} \of{t,Z_t,A} \dd{t} 
			\period
		\end{align}
	The equality is to be understood in $L^2_A$.
\end{theorem}

\begin{proof}
	Before we begin the proof, 
	we notice that both Pettis integrals on the right-hand side exist by Proposition~\ref{result:pettisExist}, together with Lemma~\ref{result:regularityVar} and Lemma~\ref{result:GBound}. 
	As before, we write $Z = \sqrt{A}Y$, where the process $Y$ is given by~\eqref{eq:Y}. 
	We prove the identity by showing that the $S_A$-transforms of both sides coincide for all $\eta \in D_A$. 
	For that we use the representation
		\begin{align*}
			S_A\bra{{F\of{t,Z_t,A}}}\of{\eta} = 
			\E\bra{\vphantom{\Big\vert}\of{S\bra{{F\of{t,\sqrt{a}Y_t,a}}}\of{\sqrt{a}\eta}}_{a = A}
			\productspacing
			\varepsilon_A^\eta\of{A}} 
			\period
		\end{align*}
	The $S$-transform $S\bra{{F\of{t,\sqrt{a}Y_t,a}}}\of{\sqrt{a}\eta}$ is, 
	as shown in the proof of Theorem~5.3 in~\cite{B03b}, differentiable with respect to $t$ on $\of{0,T}$ with derivative
		\begin{align*}
			S\bra{\pardiffQ{F}{t} \of{t,\sqrt{a}Y_t,a}}\of{\sqrt{a}\eta} 
			&+ 
			\of{M_+^H\eta}\of{t} 
			\productspacing 
			\nu\of{t} 
			\productspacing 
			S\bra{a \pardiffQ{F}{z} \of{t,\sqrt{a}Y_t,a}}\of{\sqrt{a}\eta}  
			\\ &+
			\frac{1}{2} \diff{t} \norm{M_{-}^H\of{\nu\1_{\of{0,t}}}}^2_{2}
			\productspacing 
			S\bra{a \pardiffQ[2]{F}{z} \of{t,\sqrt{a}Y_t,a}}\of{\sqrt{a}\eta} 
			\period
		\end{align*}
	From Corollary~\ref{result:GCont}, Remark~\ref{result:SofGBound} and Lemma~\ref{result:regularityVar},
	we can infer that this derivative is continuous on $\of{0,T}$ for all $0 < a$. 
	As $F$, $\pardiffQ{F}{t}$, $a \pardiffQ{F}{z}$ and $a \pardiffQ[2]{F}{z}$ are all admissible functionals of $Z$, 
	Remark~\ref{result:SofGBound} furthermore implies, 
	that the absolute value of the derivative is bounded,
	locally independent of $t$, by $C \e{\of{c + \frac{1}{2}\norm{\eta}^2_{2}}a}$, 
	where $C$ and $c$ are some positive constants. 
	As 
		\begin{align*}
			C \e{\of{c + \frac{1}{2}\norm{\eta}^2_{2}}a} 
			\productspacing
			\varepsilon_A^\eta\of{a}
		\end{align*}
	is integrable with respect to $\Pb^A$ (again see Remark~\ref{result:SofGBound}), 
	we can differentiate $S_A\bra{{F\of{t,Z_t,A}}}\of{\eta}$ on $\of{0,T}$, and the continuous derivative is given by
		\begin{align*}
			 \diff{t} S_A\bra{{F\of{t,Z_t,A}}}\of{\eta} 
		 	=
		 	\E\bra{\of{\diff{t} S\bra{{F\of{t,\sqrt{a}Y_t,a}}}\of{\sqrt{a}\eta}}_{a = A}
		 	\productspacing
		 	\varepsilon_A^\eta\of{A}} 
		 	\period 	
		\end{align*}
	The right-hand side then evaluates to 
		\begin{align*}
			\E \Bigg[ \Bigg(
			S \Bigg[ \pardiffQ{F}{t} & \of{t,\sqrt{a}Y_t,a} \Bigg] \of{\sqrt{a}\eta} 
			+ 
			\of{M_+^H\eta}\of{t} 
			\productspacing 
			\nu\of{t} 
			\productspacing 
			S\bra{a \pardiffQ{F}{z} \of{t,\sqrt{a}Y_t,a}}\of{\sqrt{a}\eta} 
			\\ &+
			\frac{1}{2} \diff{t} \norm{M_{-}^H\of{\nu\1_{\of{0,t}}}}^2_{2} 
			\productspacing
			S\bra{a \pardiffQ[2]{F}{z} \of{t,\sqrt{a}Y_t,a}}\of{\sqrt{a}\eta} \Bigg)_{a = A}
			\productspacing
			\varepsilon_A^\eta\of{a} \Bigg]
		\end{align*}
	or
		\begin{align*}
			S_A\bra{\pardiffQ{F}{t}\of{t,Z_t,A}}\of{\eta} 
			&+
			\of{M_+^H\eta}\of{t} 
			\productspacing 
			\nu\of{t} 
			\productspacing 
			S_A\bra{A \pardiffQ{F}{z}\of{t,Z_t,A}}\of{\eta} 
			\\ &+
			\frac{1}{2} \diff{t} \norm{M_{-}^H\of{\nu\1_{\of{0,t}}}}^2_{2} 
			\productspacing			
			S_A\bra{A \pardiffQ[2]{F}{z}\of{t,Z_t,A}}\of{\eta} 
		\end{align*}
	respectively.
	As $F\of{t,Z_t,A}$ is $L^2_A$ continuous with respect to $t$, we have
		\begin{align*}
			S_A\bra{F\of{T,Z_T A} - F\of{0,0,A}}\of{\eta} 
			&= 
			\lim_{\epsilon \to 0} S_A\bra{F\of{T-\epsilon,X_{T-\epsilon},A} 
			- 
			F\of{\epsilon,X_\epsilon,A}}\of{\eta}  
			\period
		\end{align*}	
	Using the representation of the derivative of $S_A\bra{{F\of{t,Z_t,A}}}\of{\eta}$ established above, 
	the limit on the right-hand side takes the form
		\begin{align*}
			\int_0^T S_A \bigg[ \pardiffQ{F}{t} & \of{t,Z_t,A} \bigg] \of{\eta} \dd{t} 
			+
			\int_0^T \of{M_+^H\eta}\of{t} \productspacing
			\nu\of{t} 
			\productspacing
			S_A\bra{A \pardiffQ{F}{z}\of{t,Z_t,A}}\of{\eta} \dd{t} 
			\\ &+ 
			\frac{1}{2} \int_0^T \diff{t} \norm{M_{-}^H\of{\nu\1_{\of{0,t}}}}^2_{2} 
			\productspacing
			 S_A\bra{A \pardiffQ[2]{F}{z} \of{t,Z_t,A}}\of{\eta} \dd{t} 
			 \period
		\end{align*}
	Note here that the integrals above are ordinary Lebesgue integrals.
	Interchanging the $S_A$-transform and the Pettis integrals by Remark~\ref{result:pettisInter} leads us to the desired statement as, 
	by the definition of the fractional  It\^o~integral, 
	we have
		\begin{align*}
			S_A\bra{\int_0^T \nu\of{t} \productspacing \pardiffQ{F}{z} \of{t,Z_t,A} \dd{X_t}}\of{\eta}
			= 
			\int_0^T \of{M_+^H\eta}\of{t} 
			\productspacing
			\nu\of{t} 
			\productspacing
			S_A\bra{A \pardiffQ{F}{z} \of{t,Z_t,A}}\of{\eta} \dd{t}
			\period
		\end{align*}
	In particular, the above considerations show that the fractional It\^o integral exists in the first place.
\end{proof}

The following corollary is particularly helpful for the applications to generalized time-fractional evolution equations in the upcoming section.

\begin{corollary} \label{result:itoFormula_cor}
Let $T > 0$ and let $Z \in \mathfrak{DI}_T\of{X}$ result from integrating the function $\nu$.
	Consider a function $V \in \CC^2\of{\R}$ satisfying
		\begin{align} \label{eq:ItoCorCondition}
			\abs{V\of{z}} \comma \mediumspace
			\abs{\diffQ{V}{z}\of{z}}\comma \mediumspace
			\abs{\diffQ[2]{V}{z}\of{z}} 
			\lesssim
			\e{k \abs{z}}
		\end{align}
	for some $0 \leq k < \sqrt{2 R_A} \productspacing \of{2 T^H \max_{s \in \bra{0,T}}{\abs{\nu\of{s}}}}^{-1}$.
	Then, for all  $x_0 \in \R$, the function $V\of{x_0 + \cdot}$ satisfies the assumptions of Theorem~\ref{result:itoFormula} and the It\^o formula~\eqref{eq:Ito} yields
		\begin{align} \label{eq:ItoCor}
			V\of{x_0 + Z_T} 
			= 
			V\of{x_0} &+ 
			\int_0^T  \nu\of{t} 
			\productspacing
			\diffQ{V}{z} \of{x_0 + Z_t} \dd{X_t} 
			\nonumber \\ &+ 
			\frac{A}{2} \int_0^T \diff{t} \norm{M_{-}^H\of{\nu\1_{\of{0,t}}}}^2_{2} 
			\productspacing 
			\diffQ[2]{V}{z} \of{x_0 + Z_t} \dd{t} 
			\period
		\end{align}
\end{corollary}

\begin{proof}
	Let  $x_0 \in \R$. 
	As equation \eqref{eq:ItoCor} is a simple application of the It\^o~formula, it suffices to verify that $V\of{x_0 + \cdot}$ satisfies the assumptions of the formula. 
	Namely we verify that $V\of{x_0 + \cdot}$, $a\pardiffQ{V\of{x_0 + \cdot}}{z}$ and $a\pardiffQ[2]{V\of{x_0 + \cdot}}{z}$ are admissible functionals of $Z$. 
	Note that for any $x_0$ the function $V\of{x_0 + \cdot}$ again is in $\CC^2\of{\R}$ and satisfies \eqref{eq:ItoCorCondition} with the same constant $k$. 
	Hence, it is enough to consider the case $x_0 = 0$. It is then straight forward to see that
		\begin{align*}
			\abs{V\of{\sqrt{a}x}} \comma \mediumspace
			\abs{ \delimiterspacing a \diffQ{V}{z}\of{\sqrt{a}z} \delimiterspacing } \comma \mediumspace
			\abs{\delimiterspacing a \diffQ[2]{V}{z}\of{\sqrt{a}z} \delimiterspacing }
			\lesssim
			\of{a + 1} \e{k \sqrt{a} \abs{x}}	
			\period
		\end{align*}
	For any $\lambda > 0$ and any $a > 0$ we have $k\sqrt{a}\abs{x} \leq k^2a / \of{4\lambda} + \lambda x^2$. 
	Furthermore for any $\epsilon > 0$ it is  $a+1 \lesssim \e{\epsilon a}$. 
	Overall it follows that
		\begin{align*}
			\of{a + 1}\e{k\sqrt{a} \abs{x}} 
			\leq 
			\of{a + 1}\e{ \frac{k^2a}{4\lambda} + \lambda x^2} 
			\lesssim
			\e{ \of{\epsilon + \frac{k^2}{4\lambda}}a + \lambda x^2}. 
		\end{align*}
	Due to the bound on $k$ we can choose a $0 < \lambda < \of{2T^H \max_{s \in \bra{0,T}}{\abs{\nu\of{s}}}}^{-2}$ and a $\epsilon > 0$ small enough, such that $\epsilon + k^2 / \of{4\lambda} < {R_A} / 2$.
	Thus $V$ satisfies the assumptions of the Ito formula.
\end{proof}

We end up this section with two examples of the It\^o~formula.

\begin{example}
	Consider $V(z) \coloneqq  z^2$ and note that $X\in \mathfrak{DI}_\infty\of{X}$ by integrating $\nu = 1$. 
	Thus, by Corollary~\ref{result:itoFormula_cor},
		\begin{align*}
			X^2_t 
			= 
			\int_0^t 2 X_s \dd{X_s} + \frac{A}{2}\int_0^t \diff{s} \norm{M_{-}^H\of{\nu\1_{\of{0,s}}}}^2_{2} \dd{s} \comma \largespace 
			t \geq 0.
	\end{align*}
	Since the second integral trivially evaluates to 
		\begin{align*}
			\int_0^t \diff{s} \norm{M_{-}^H\of{\nu\1_{\of{0,s}}}}^2_{2} \dd{s} =			 \norm{M_{-}^H\of{\nu\1_{\of{0,t}}}}^2_{2} 
			= 
			t^{2H}
			\comma
		\end{align*}		
	we have for all $t \geq 0$
		\begin{align*}
			\int_0^t X_s \dd{X_s} 
			= 
			\frac{1}{2} \of{X^2_t - A t^{2H}}
			\period
		\end{align*}
\end{example}

\begin{example}
	Let $g_0 > 0$ and $\nu$, $r$, $\ffi\,:\,[0,\infty)\to\R$ be some Borel functions.
	We call
		\begin{align*}
			\GG_t 
			\coloneqq 
			g_0\exp\of{\ffi(A) \int_0^t r(s)\dd{s} - \frac{A}{2} \diff{t} \norm{M_{-}^H\of{\nu\1_{\of{0,t}}}}^2_{2} + \int_0^t \nu(s) \dd{X_s}}
		\end{align*}
	a \emphasis{geometric randomly scaled Gaussian process} with coefficients $H$, $g_0$, $\nu$, $r$ and $\ffi$, provided the right-hand side exists as an element of $L^2_A$ for all $t \in [0,\infty)$.  
	If furthermore $\limsup_{a \to \infty} \ffi\of{a}/\sqrt{a} \leq 1$, $r$ is continuous  and 
		\begin{align*}
			Z_t 
			\coloneqq 
			\int_0^t \nu(s) \dd{X_s} \comma \largespace
		 	t \in \left[0,\infty \right) 
			\comma
		\end{align*}	
	belongs to the class $\mathfrak{DI}_\infty\of{X}$, then we have  $\GG_t = F\of{t,Z_t,A}$ for the admissible (for each $T > 0$) functional 
		\begin{align*}
			F\of{t,z,a} 
			\coloneqq
 			g_0 \exp\of{\ffi\of{a}\int_0^t r\of{s} \dd{s} - \frac{a}{2} \diff{t} \norm{M_{-}^{H}\of{\nu\1_{\of{0,t}}}}^2_{2} + z} \comma
		\end{align*}
	and hence $\of{\GG_t}_{t \geq 0}$ exists. 
	Moreover, by the It\^o~formula~\eqref{eq:Ito}, 
	the process $(\GG_t)_{t\geq0}$ solves the following stochastic integral equation
		\begin{align*}
			\GG_t
			=
			g_0 + \ffi\of{A} \int_0^t r\of{s} \productspacing \GG_s \dd{s} 
			+
			\int_0^t \nu\of{s} \productspacing	\GG_s \dd{X_s}.
		\end{align*}
	Furthermore, for some $T > 0$, assume that $\Phi \in \CC^{1,2,0}\of{\bra{0,T} , \of{0,\infty} , \of{0,\infty}}$ is such that $\Phi\of{t,g,a}$, $\pardiff{t} \Phi\of{t,g,a}$, $\pardiff{g} \Phi\of{t,g,a}$, $\pardiff[2]{g} \Phi\of{t,g,a}$ are of at most polynomial growth w.r.t. $g$, $a$ uniformly w.r.t. $t \in \bra{0,T}$. 
	Then $\Phi\of{t,F\of{t,z,a},a}$ is an admissible functional for $Z$ and, again by the It\^o~formula~\eqref{eq:Ito},  the fractional It\^o~integral 
		\begin{align*}
			\int_0^T \nu\of{t} \productspacing \GG_t \pardiffQ{\Phi}{g}\of{t,\GG_t,A} \dd{X_t}
		\end{align*}
	exists and is equal to 
		\begin{align*}
			\Phi\of{T,\GG_T,A} &- \Phi\of{0,g_0,A} - \int_0^T \pardiffQ{\Phi}{t}\of{t,\GG_t,A} \dd{t}
			-
			\ffi\of{A} \int_0^T r\of{t} \productspacing \GG_t \productspacing \pardiffQ{\Phi}{g}\of{t,\GG_t,A} \dd{t} 
			\\ &-
			\frac{A}{2} \int_0^T \diff{t} \norm{M_{-}^{H}\of{\nu\1_{\of{0,t}}}}^2_{2}  \productspacing \GG_t^2  \productspacing \pardiffQ[2]{\Phi}{g}\of{t,\GG_t,A} \dd{t}
			\period
		\end{align*}
\end{example}

\section{Randomly Scaled Gaussian Processes and Related  Evolution Equations}
In this section, we investigate how RSGPs as in Theorem~\ref{result:itoFormula} can be used for solving evolution equations with fractional operators with respect to time. 
It has been shown in the works~\cite{BBB22,BB22} that a class of generalized time fractional evolution equations can be solved by Feynman-Kac type formulas with the use of randomly scaled FBM. 
\par
Now (as before) let $X \coloneqq X^{A,H} \coloneqq \sqrt{A}B^H$, where $B^H \coloneqq \of{B^H_t}_{t \in \left[0,\infty\right)}$ be a FBM, $H\in(0,1)$, 
$A$ be a positive random variable independent from $B^H$ such that its Laplace transform $\laplace\bra{A}$ is holomorphic on some open ball $B_{R_A}(0)$ with maximal radius $R_A > 0$. 
Let the process 
	\begin{align*}
		Z_t 
		\coloneqq
		\int_0^t \nu\of{s} \dd{X_s} \comma \largespace 
		t \in \left[0,\infty\right)
	\end{align*}
belong to the class $\mathfrak{DI}_\infty$ and note that the variance of this process is given by the differentiable function 
	\begin{align*}
		t \mapsto 
		\E[A]\,\norm{M_{-}^{H}\of{\nu\1_{\of{0,t}}}}^2_{2}
	\end{align*}
by  Lemma~4.2. We assume that $\nu$ is such that this variance is a monotonically increasing function. This is the case, e.g., when $\nu$ does not change its sign (see formula~\eqref{eq:new}).
Now let $u_0\in \Sch(\R)$.
We consider
	\begin{align}\label{eq:def_u}
		u \smallspace : \smallspace 
		\left[0,\infty\right) \times \R \to \R \comma \largespace 
		\of{t,x} \mapsto \E\bra{u_0\of{x + Z_t}}.
\end{align}
Applying the It\^o~formula as in Corollary~\ref{result:itoFormula_cor} and taking expectation results in
	\begin{align}\label{eq:itoApp}
		u\of{t,x} 
		=
		u_0\of{x} + \frac{1}{2} \int_0^t \diff{t}\norm{M_{-}^{H}\of{\nu\1_{\of{0,t}}}}^2_{2}
		\productspacing
		\pardiff[2]{x} 
		\E\bra{A u_0\of{x+Z_s}} \dd{s} \comma \largespace
	 	x \in \R \comma \largespace 
	 	t > 0
	 	\period
\end{align}
We wish to clarify the relation between $u\of{t,x} = \E\bra{u_0\of{x + Z_t}}$ and
	\begin{align}\label{eq:u-tilde}
		\tilde{u}\of{t,x}
		\coloneqq
		\E\left[ Au_0(x+Z_t) \right],\qquad t>0,\largespace x\in\R.
	\end{align}
Of course, this relation depends on the diffusion coefficient $A$. 
Below we consider some particular cases of this relation.

\begin{theorem} \label{result:evoEq_gen}
  In the above setting, consider a Borel-measurable mapping
  \begin{align}
    \KK \smallspace : \smallspace
    \of{0,\infty} \times \of{0,\infty} \times \R \to \R \comma \largespace \of{t,s,\xi} \mapsto \KK_{t,s}\of{\xi}
  \end{align}
  and suppose that the Laplace transform of the diffusion coefficient $A$ satisfies the evolution equation
  		\begin{align}\label{eq:L[A]}
    		\laplace\bra{A}\of{t^{2H} \xi} 
    		= 
    		1 + \int_0^t \KK_{t,s}\of{\xi}
    		\productspacing 
    		\laplace\bra{A}\of{s^{2H} \xi} \dd{s} \comma \largespace  
    		t , \xi > 0 
    		\period
	  \end{align}
	Furthermore, suppose that for all fixed $t , s \in \of{0,\infty}$ with $t \geq s$ the function $ \xi \mapsto \abs{\KK_{t,s}\of{\xi^2/2}}$ is in $\mathcal{S}^{\prime}\of{\R}$ and for all fixed $t \in \of{0,\infty}$ and $f \in \mathcal{S}\of{\R}$ we have
    \begin{align} \label{eq:rhoCondition}
      \abs{\KK_{t, \cdot}\of{\xi^2/2}}\of{f} = \int_{\R} \abs{\KK_{t, \cdot}\of{\xi^2/2}}f\of{\xi} \dd{\xi} \in L^1\of{0,t}
      \period
    \end{align}
  Consider also the pseudo-differential operator $\KK_{t,s}\of{- \Delta / 2}$ with symbol $\KK_{t,s}\of{\xi^2/2}$, 
  i.e. (with the Fourier transform $\fourier$)
  	\begin{align*}
  		\KK_{t,s}\of{- \Delta / 2} f\of{x}
    &\coloneqq
    \of{\fourier^{-1}_{\xi \mapsto x} \circ \KK_{t,s}\of{\xi^2/2} \circ \fourier_{x\mapsto\xi}}f\of{x}
    \\ &\equiv
    \frac{1}{2\pi} \int_{\R} \int_\R \e{i\of{x-q}\xi}
		\productspacing    
    \KK_{t,s}\of{\xi^2/2} 
		\productspacing    
    f\of{q} \dd{q} \dd{\xi}.
  \end{align*}
  Then for any $u_0 \in \Sch\of{\R}$ the function $u\of{t,x} = \E\bra{u_0\of{x + Z_t}}$, as defined in formula~\eqref{eq:def_u}, is a solution of the evolution equation
  \begin{align}\label{eq:eveq1}
    u\of{t,x} = 
    u_0\of{x} + 
    \int_0^t \sigma_\nu^\prime\of{s} 
		\productspacing    
    \KK_{\sigma_\nu\of{t},\sigma_\nu\of{s}}\of{- \Delta / 2}u\of{s,x} \dd{s} \comma \largespace 
    t > 0  \comma \largespace 
    x \in \R 
    \comma
  \end{align}
  where the time change $\sigma_\nu$ is given
  by
  \begin{align}\label{eq:sigma}
    \sigma_\nu(t)
    \coloneqq
    \norm{M_{-}^{H}\of{\nu\1_{\of{0,t}}}}^{1/H}_{2}
    \period
  \end{align}
  The solution $u$ is continuous in $t$ and infinitely often differentiable with respect to $x$.
\end{theorem}

\begin{proof}
	Due to $u_0 \in \mathcal{S}\of{\R}$ the functions $\tilde{u}$ and $u$ are infinitely often differentiable with respect to the space variable $x$. And all these partial  derivatives  with respect to the space variable $u^{(m)}$ satisfy: $u^{\of{m}}\of{t, \cdot}$ and $\tilde{u}^{\of{m}}\of{t, \cdot}$ are $L^1$ functions;  $u^{\of{m}}\of{\cdot, x}$ and $\tilde{u}^{\of{m}}\of{\cdot, x}$ are continuous on $\of{0,\infty}$ for all $x \in \R$ and $m \geq 0$. 
	Therefore, the Fourier transforms of $\tilde{u}\of{t,\cdot}$ and $u\of{t,\cdot}$ exist for all $t > 0$ and can be calculated to be 
		\begin{align}\label{eq:fourier_calc}
			\fourier\bra{u\of{t, \cdot}}\of{\xi} 
			&=
			\of{2 \pi}^{-1/2} \int_\R \e{-i \xi x} 
			\productspacing
			\E\bra{u_0\of{x + Z_t}} \dd{x} 
			\nonumber \\ &=
			\E\bra{\of{2 \pi}^{-1/2} \int_\R \e{-i \xi \of{x - Z_t}} 
			\productspacing
			u_0\of{x} \dd{x}} 
			\nonumber \\ &= 
			\fourier\bra{u_0}\of{\xi} 
			\productspacing			
			\E\bra{\e{- A \sigma_\nu^{2H}\of{t} \frac{\xi^2}{2}}} 
			=
			\fourier\bra{u_0}\of{\xi} 
			\productspacing
			\laplace\bra{A}\of{\sigma_\nu^{2H}\of{t} \frac{\xi^2}{2}} 
	\end{align}
	and, similarly, 
		\begin{align*}
			\fourier\bra{\tilde{u}\of{t, \cdot}}\of{\xi} 
			=	 
			\fourier\bra{u_0}\of{\xi} 
			\productspacing
			\E\bra{A \e{-A \sigma_\nu^{2H}\of{t} \frac{\xi^2}{2}}} 
			=
			-\fourier\bra{u_0}\of{\xi} 
			\productspacing
			\laplace\bra{A}^{\prime}\of{\sigma_\nu^{2H}\of{t} \frac{\xi^2}{2}} 
			\period
		\end{align*}
	Using this and due to~\eqref{eq:L[A]}, the Fourier transform of the integral on the right-hand side of equation~\eqref{eq:itoApp} can be calculated to be
		\begin{align} \label{eq:fourierTrafoFunctions}
			\fourier\Bigg[ \int_0^t \of{\sigma_\nu^{2H}}^\prime\of{s}
			\productspacing
			& \frac{1}{2}\pardiff[2]{x} \tilde{u}\of{s,x} \dd{s} \Bigg]\of{\xi}
     	\nonumber \\ &= 
     	\fourier\bra{u_0}\of{\xi} \int_0^t \of{\sigma_\nu^{2H}}^\prime\of{s} 
     	\productspacing
    		\frac{\xi^2}{2} 
    		\productspacing
     	\laplace\bra{A}^\prime\of{\sigma_\nu^{2H}\of{s} \frac{\xi^2}{2}} \dd{s} 
			\nonumber \\ &=
    		\fourier\bra{u_0}\of{\xi} \of{\laplace\bra{A}\of{\sigma_\nu^{2H}\of{t} \frac{\xi^2}{2}} - 1}
			\nonumber \\ &=
    	\fourier\bra{u_0}\of{\xi} \int_0^{\sigma_\nu\of{t}} \KK_{\sigma_\nu\of{t},s}\of{\xi^2 / 2} 
    		\productspacing
    		\laplace\bra{A}\of{s^{2H} \frac{\xi^2}{2}} \dd{s}
			\nonumber \\ &=
    		\fourier\bra{u_0}\of{\xi}\int_0^t \sigma_\nu^\prime\of{s}
   		\productspacing
    		\KK_{\sigma_\nu\of{t},\sigma_\nu\of{s}} \of{\xi^2 / 2} 
    		\productspacing
    		\laplace\bra{A}\of{\sigma_\nu^{2H}\of{s} \frac{\xi^2}{2}} \dd{s}
   		\period
		\end{align}
  Note at this point that $\fourier\bra{u_0} \in \mathcal{S}\of{\R}$ and that
   	\begin{align*}
    		\of{s, \xi}	\mapsto \sigma^{\prime}_\nu\of{s} 
    		\productspacing
    		\laplace\bra{A}\of{\sigma_\nu^{2H}\of{s}\frac{\xi^2}{2}}
   	\end{align*}
 	is bounded on $\bra{0,t} \times \R$. 
  	Hence, by assumption~\eqref{eq:rhoCondition} the function
    \begin{align}\label{eq:transformedFunction}
  			\fourier\bra{u_0}\of{\xi}
  			\productspacing
  			\sigma_\nu^\prime\of{s}
  			\productspacing
  			\KK_{\sigma_\nu\of{t},\sigma_\nu\of{s}}\of{\xi^2 / 2} 
  			\productspacing
  			\laplace\bra{A}\of{\sigma_\nu^{2H}\of{s} \frac{\xi^2}{2}}
   	\end{align}
	is an element of $L^1\of{\bra{0,t} \times \R}$ and, in particular, the functions in~\eqref{eq:fourierTrafoFunctions} is in $L^1\of{\R}$ with respect to $\xi$.
	The application of the inverse Fourier transform to \eqref{eq:fourierTrafoFunctions} therefore yields
		\begin{align*}
    		\int_0^t \of{\sigma_\nu^{2H}}^\prime\of{s} 
    		\productspacing
    		& \frac{1}{2}\pardiff[2]{x} \E\bra{A u_0\of{x + Z_s}} \dd{s}
      	\\ &=
      	\fourier^{-1}\bra{\fourier\bra{u_0}\of{\xi}\int_0^t \sigma_\nu^\prime\of{s} 
      	\productspacing
      	\KK_{\sigma_\nu\of{t},\sigma_\nu\of{s}} \of{\xi^2 / 2} 
      	\productspacing
      	\laplace\bra{A}\of{\sigma_\nu^{2H}\of{s} \frac{\xi^2}{2}} \dd{s}}\of{x}
    		\\ &= 
     	\int_0^t \sigma_\nu^\prime\of{s} 
     	\productspacing
     	\KK_{\sigma_\nu\of{t},\sigma_\nu\of{s}}\of{-\Delta / 2} \E\bra{u_0\of{x + Z_s}} \dd{s}
   		\period
	\end{align*}
	Together with equation~\eqref{eq:itoApp} this shows the statement.
\end{proof}

\begin{example}[Gamma-Grey Brownian Motion]
	Let first $\nu \equiv 1$, i.e.  $Z_t = X_t = \sqrt{A}B^H_t$, $t \geq 0$, $H \in \of{0,1}$. Note that $\operatorname{Var}\bra{B^H_t} = t^{2H}$. 
	Let $\rho \in \of{0,1}$ and the diffusion coefficient $A \coloneqq A_\rho$ have the following probability density function
		\begin{align*}
			f_\rho\of{a}
			\coloneqq
			\frac{1}{\Gamma\of{\rho}\Gamma\of{1-\rho}a\of{a-1}^{\rho}}\1_{\of{1,\infty}}\of{a}.
		\end{align*}
	Then the process $\of{X_t}_{t \geq 0}$ coincides in distribution with gamma-grey Brownian motion considered in~\cite{BCG23} (cf. the proof of Theorem~4.1 in~\cite{BCG23}). 
	Note that the characteristic function of $\of{X_t}_{t\geq0}$ is given by
		\begin{align*}
			\E\left[ \e{i\lambda X_t} \right]
			=
			\E\left[ \e{-A_\rho t^{2H}\lambda^2/2 }  \right]
			=
			\laplace[A_\rho]\left(t^{2H}\lambda^2/2\right)
			=
			\frac{\Gamma(\rho,t^{2H}\lambda^2/2)}{\Gamma\of{\rho}},
	\end{align*}
	where $\Gamma\of{\rho,x} \coloneqq \int_x^\infty \e{-w}w^{\rho-1} \dd{w}$ is the upper incomplete gamma function, $\Gamma\of{\rho} \coloneqq \Gamma\of{\rho,0}$. 
	Since $\Gamma\of{\rho,\cdot}$ is an entire function for $\rho \in \of{0,1}$, the diffusion coefficient $A_\rho$ satisfies our assumptions.
	 Moreover, by Theorem~5.2 of~\cite{BCG23}, we have for all $t \geq 0$, $\xi > 0$
		\begin{align*}
			\laplace\bra{A_\rho} & \of{t^{2H}\xi}
			\\&=
			1 - 2H \int_0^t \xi \e{-\xi\of{t^{2H} - s^{2H}}}\e[\rho]{\xi\of{t^{2H} - s^{2H}}} s^{2H-1}\laplace\bra{A_\rho}\of{s^{2H}\xi} \dd{s}
			\comma
		\end{align*}
	where $\e[\rho]z \coloneqq z^{\rho-1}E_{\rho,\rho}\of{z^\rho}$ is the so-called $\rho$-exponential function (see~\cite{KST06}, p.~50, formula~(1.10.11)).
	Thus, equality~\eqref{eq:L[A]} is true with
		\begin{align*}
			\KK_{t,s}\of{\xi}
			\coloneqq
			-2Hs^{2H-1} \xi \e{-\xi\of{t^{2H} - s^{2H}}}\e[\rho]{\xi\of{t^{2H} - s^{2H}}}. 
		\end{align*}
	With the application of Theorem~\ref{result:evoEq_gen} in mind we now show that, for all $t \geq s$, the above kernel $\KK_{t,s}$ satisfies $\abs{\KK_{t,s}\of{\xi^2/2}} \in \Sch^{\prime}\of{\R}$,
	as well as \eqref{eq:rhoCondition} for all $t \geq 0$ and $f \in \Sch\of{\R}$.
	Regarding the first property we note that the asymptotic behavior of the Mittag-Leffler function $E_{\rho,\rho}\of{\xi^\rho}$ for large positive arguments is know to be
		\begin{align*}
  			E_{\rho,\rho}\of{\xi^\rho} 
  			= 
  			\xi^{1 - \rho} \productspacing\frac{\e{\xi}}{\rho} + \mathcal{O}\of{\xi^{-\rho}} \comma \largespace 
  			\xi \to \infty 
  			\comma
		\end{align*}
	see equation~(6.10) in~\cite{HMS11}. 
	Hence it is 
		\begin{align}
			\e{-\xi} E_{\rho,\rho}\of{\xi^\rho} 
			\lesssim 
			1 + \xi^{1 - \rho}
		\end{align}	
	for $\xi \geq 0$ and thus for all $t$ and $s$ with $t \geq s$ we obtain
		\begin{align} \label{eq:exponentialTerm}
	 		\abs{\KK_{t, s}\of{\xi^2/2}} 
	 		&= 
	 		H s^{2H - 1} \xi^2 \e{-\xi^2\of{t^{2H}-s^{2H}} / 2}\e[\rho]{\xi^2\of{t^{2H}-s^{2H}} / 2}
	 		\nonumber \\ &\lesssim 
	 		\xi^{2 \rho} \of{t^{2H} - s^{2H}}^{\rho - 1} s^{2H - 1} \of{1 + \xi^{2\of{1 - \rho}}t^{2H\of{1 - \rho}}} \comma
		\end{align}
	where we used the estimate $t^{2H} - s^{2H} \leq t^{2H}$ in the last bracket.
	This implies that $\KK_{t,s}\of{\xi^2/2}$ grows at most polynomialy with respect to $\xi$ for all $t \geq s$,
	and hence $\abs{\KK_{t,s}\of{\xi^2/2}} \in \mathcal{S}^{\prime}\of{\R}$.
	Concerning the second property, we show that the integral
		\begin{align}
			\int_{0}^t \int_{\R} \abs{\KK_{t,s}\of{\xi^2/2} f\of{\xi}}  \dd{\xi} \dd{s}
		\end{align}
	is finite for every $t \geq 0$ and $f \in \mathcal{S}\of{\R}$. 
	By writing
		\begin{align*}
			\gamma_t\of{\xi} 
			= 
			\xi^{2p}\of{1 + \xi^{2\of{1 - \rho}}t^{2H\of{1 - \rho}}}
		\end{align*}
	and using estimate	\eqref{eq:exponentialTerm} one may directly find
		\begin{align*}
			\int_{0}^t \int_{\R} & \abs{\KK_{t,s}\of{\xi^2/2} f\of{\xi}}  \dd{\xi} \dd{s}
			\\ &\lesssim
  			\int_{0}^t \int_{\R} s^{2H - 1}\of{t^{2H} - s^{2H}}^{\rho - 1} \abs{f\of{\xi}} 
  			\productspacing
  			\gamma_t\of{\xi}  \dd{\xi} \dd{s}
  			\\ &=
     	\of{\int_{0}^t s^{2H - 1}\of{t^{2H} - s^{2H}}^{\rho - 1} \dd{s}} \of{\int_{\R} \abs{f\of{\xi}} 
			\productspacing     	
     	\gamma_t\of{\xi} \dd{\xi}}
     	<
     	\infty
     	\period
		\end{align*}
	Therefore, Theorem~\ref{result:evoEq_gen} is indeed applicable to the current setting and the governing equation for the gamma-grey noise is given by~\eqref{eq:eveq1} with the corresponding operator $\KK_{t,s}\of{-\Delta / 2}$ and $\sigma_\nu\of{t} \coloneqq t$. 
	For general $\nu$, such that the process $Z_t:=\int_0^t\nu(s)dX_s$, $t\geq0$, belongs to the class  $\mathfrak{DI}_\infty$, 
	the governing equation for the process $Z$ is given by~\eqref{eq:eveq1} with the same operator $\KK_{t,s}\of{-\Delta / 2}$ but with the time change  $\sigma_\nu\of{t} = \norm{M_-^H\of{\nu\1_{\of{s,t}}}}^{1/H}_{L^2(\R)}$.
\end{example}

\begin{example}[The Setting of Bender-Butko \cite{BB22}] 
	Let $k : \of{0,\infty} \times \of{0,\infty} \to \R$ be a Borel-measurable function  which is homogeneous of degree $\theta-1$ for some $\theta > 0$,
	i.e. $k\of{t,ts} = t^{\theta-1}k\of{1,s}$ for all $t > 0$ and $s \in (0,1)$. 
	Let also $k\of{1,\cdot}\in L^{1+\epsilon}\of{0,1}$ for some $\epsilon > 0$. 	
	Then, by Theorem~2 and Corollary~1 of~\cite{BB22}, the function 
		\begin{align} \label{eq:def_Phi}
			\Phi_k\of{t^\theta z}
			\coloneqq 
			\sum_{n=0}^\infty c_n \of{-z}^n \comma \largespace 
			c_0 \coloneqq 1 \comma \largespace 
			c_n \coloneqq c_{n-1}\int_0^1 k\of{1,s} s^{\theta(n-1)} \dd{s} \comma \largespace
			 n \in \N 
			 \comma
\end{align}
	is an entire function and satisfies the equality
		\begin{align}\label{eq:cor_Bebu}
			\Phi_k\of{t^\theta z}
			=
			1 - z \int_0^t k\of{t,s} \Phi_k\of{s^\theta z} \dd{s} \comma \largespace
			t > 0 \comma \largespace 
			z \in \C
			\period
\end{align}
	Assume additionally that the function $\Phi_k$ is completely monotone.
	Then there exists a nonnegative random variable $A$ such that $\laplace\bra{A} = \Phi_k$. 
	Then, if $\theta\in(0,2)$, equality~\eqref{eq:cor_Bebu} is just a special case of~\eqref{eq:L[A]} with $\KK_{t,s}\of{\xi} \coloneqq -k\of{t,s}\xi$ and $H \coloneqq \theta/2$.
  Note that because of $k\of{1, \cdot} \in L^{1 + \epsilon}\of{0,1} \subset L^{1}\of{0,1}$ this choice of $\KK_{t,s}$ is covered by the assumptions of Theorem~\ref{result:evoEq_gen}.
	It has been shown in~\cite{BB22} that  the function
		\begin{align*}
			u\of{t,x}
			\coloneqq
			\E\bra{u_0\of{x + \sqrt{A}B^H_t}}		
		\end{align*}
	solves the following generalized time fractional evolution  eqution (see Theorem~2 and discussion before Remark~2 in~\cite{BB22}):
	\begin{align*}
		u\of{t,x} = 
		u_0\of{x} + \frac{1}{2} \int_0^t k\of{t,s} \pardiff[2]{x} u\of{s,x} \dd{s}.
	\end{align*}
In particular, for $k\of{t,s} \coloneqq \frac{\alpha}{\beta \Gamma\of{\beta}} s^{\frac{\alpha}{\beta}-1} \of{t^{\frac{\alpha}{\beta}} - s^{\frac{\alpha}{\beta}}}^{\beta-1}$, $\beta \in \left(0,1\right]$, $\alpha \in \of{0,2}$, the corresponding process $\left(\sqrt{A}B^{H}_t\right)_{t \geq 0}$ with $H \coloneqq \alpha/2$, is a \emphasis{generalized grey Brownian motion}.
\end{example}

\begin{corollary} \label{result:evoEq}
	Let $u_0$ be in $\Sch\of{\R}$ and let $k$ be a homogeneous kernel of degree $2H-1 \in \of{-1,1}$, satisfying $k\of{1, \cdot} \in L^{1+\epsilon}\of{0,1}$ for some $\epsilon > 0$.
	Assume that the corresponding function $\Phi_k$, as defined in~\eqref{eq:def_Phi}, is completely monotone and the Laplace transform of $A$ is given by $\laplace\bra{A} = \Phi_k$.
	Let  $Z \in \mathfrak{DI}_\infty\of{X}$ be  given by integrating $\nu$, such that the function $\sigma_\nu$ in~\eqref{eq:sigma} is monotonically increasing. Then  the function $u$ as in formula~\eqref{eq:def_u} 	
	is a continuous  solution of the evolution equation
		\begin{align*}
			u\of{t,x} 
			= 
			u_0\of{x} + \int_0^t k\of{\sigma_\nu\of{t},\sigma_\nu\of{s}}
			\productspacing			
			\sigma_\nu^\prime\of{s} 
			\productspacing
			\frac{1}{2} \pardiff[2]{x} u\of{s,x} \dd{s} \comma \largespace 
			t > 0 \comma \largespace 
			x \in \R 
			\period
		\end{align*}
\end{corollary}

To make also a statement about homogeneous kernels of degree $\gamma \coloneqq \theta-1$ with $\theta > 2$, we observe the following:

\begin{lemma} \label{result:timeChange}
  Let $1/2 < H < 1$ and $\beta \geq 0$. 
  For $C \in \R$ consider $\nu\of{t} \coloneqq C t^\beta$. 
  Then
		\begin{align*}
			\norm{M^H_-\of{\nu\1_{\of{0,t}}}}_{2}^{1 / H} 
			= 
			C_{H, \beta} 
			\productspacing
			C^{1 / H}
			\productspacing
			t^{1 + \beta / H}  \comma \largespace
			t > 0
			\comma 
		\end{align*}
	where $C_{H, \beta}$ 
	is a constant, independent of $C$, which may be expressed via the beta function $\mathrm{B}$ as
		\begin{align*}
			C_{H, \beta} \coloneqq 
			\of{\frac{2H\of{2H-1}}{2\beta + 2H} 
			\productspacing			
			\mathrm{B}\of{\beta + 1, 2H - 1}}^{\frac{1}{2H}}
			\period
		\end{align*}
\end{lemma}

\begin{proof}
	By an identity provided in Proposition~2.2 in~\cite{GN96} we have for all $t > 0$ that
		\begin{align*}
    		\norm{M^H_-\of{\nu\1_{\of{0,t}}}}_{2}^2
     	= 
     	2H\of{2H-1}C^2\int_0^t \int_0^\tau s^\beta\tau^\beta \abs{s - \tau}^{2H-2} \dd{s} \dd{\tau} 
     	\period
		\end{align*}
	Using the substitution $s \to \tau s$, the inner integral can be seen to evaluate to
 		\begin{align*} 
			\tau^{2\beta + 2H - 1} \int_0^1 s^{\beta} \of{1 - s}^{2H - 2} \dd{s}
     	=
    		\tau^{2\beta + 2H - 1} \mathrm{B}\of{\beta + 1, 2H - 1}
   		\comma
  		\end{align*}
  	where $\mathrm{B}$ denotes the beta function. 
  	For details on the beta function see~\cite{GR07}, in particular subsection~3.38.
	Due to $\beta + 1 > 0$ and $2H - 1 > 0$ the above value of the beta function is finite.
	The outer integral can then be trivially computed to be 
		\begin{align*}
			\frac{2H\of{2H-1}c^2}{2\beta + 2H}
			\productspacing			
			\mathrm{B}\of{\beta + 1, 2H - 1}
			\productspacing
			t^{2\beta + 2H}
			\period
		\end{align*}
	Overall, this gives us
		\begin{align*}
			\norm{M^H_- \1_{\of{0,t}}\nu}_{2}^{1 / H} 
			= 
			C_{H, \beta} 
			\productspacing
			C^{1 / H}
			\productspacing
			t^{1 + \beta / H}
			\period
	\end{align*}
	\end{proof}

\begin{corollary}
	Let $u_0$ be in $\mathcal{S}\of{\R}$ and
	let $k$ be a homogeneous kernel with degree $\gamma > 1$. 
	Assume $k\of{1,\cdot} \in L^{1+\epsilon}\of{0,1}$ for some $\epsilon > 0$.
	Furthermore assume that the corresponding function $\Phi_k$, as defined in~\eqref{eq:def_Phi}, is completely monotone and the Laplace transform of $A$ is given by $\laplace\bra{A} = \Phi_k$.	
	Finally assume that $H = \of{1+\gamma}/ \of{2\gamma} \in \of{1/2,1}$. 
	Then the function 
		\begin{align}\label{eq:u_cor2}
			u \smallspace : \smallspace
			\left[0,\infty\right) \times \R \to \R \comma \largespace 
			\of{t,x} \mapsto \E\bra{u_0\of{x + \int_0^t C s^{\beta} \dd{X_s}}} 
		\end{align}
	with
    	\begin{align*}
      	\beta 
      	= 
      	\frac{1}{2}\of{\gamma - \gamma^{-1}} > 0
      	\largespace \textnormal{and} \largespace
   		C
      	=
    		\of{\mathrm{B}\of{\beta + 1, 2H - 1} H}^{-1/2}
      	=
      	C_{H,\beta}^{-H} 
   		\end{align*}
	is a continuous solution of the evolution equation
	\begin{align*}
		u\of{t,x} 
		= 
		u_0\of{x} + \int_0^t k\of{t,s} 
		\productspacing 
		\frac{1}{2} \pardiff[2]{x} u\of{s,x} \dd{s} \comma \largespace 
		t > 0  \comma \largespace 
		x \in \R 
		\period
	\end{align*}
\end{corollary}

\begin{proof}
	We consider the kernel $\kappa$ given by 
		\begin{align*}
			\kappa\of{t,s} 
			\coloneqq 
			\frac{k\of{t^{1 / \gamma},s^{1 / \gamma}}}{\gamma s^{\of{\gamma - 1}/ \gamma}} 
			\period
		\end{align*}
	By construction, $\kappa$ is homogenous with degree $\gamma^{-1} \in \of{0,1}$.
	Using the substitution $s^{\frac{1}{\gamma}}$, together with the H\"older inequality, we can see that for $\delta \in \of{0,\epsilon}$ small enough it is 
		\begin{align*}
			\int_0^1 \abs{\kappa\of{1,s}}^{1 + \delta} \dd{s} 
			&= 
			\gamma^{-\delta} \int_0^1 \abs{k\of{1,s}}^{1 + \delta} s^{\delta\of{1 - \gamma}} \dd{s} 
			\\ &\leq 
			\gamma^{-\delta} \of{\int_0^1 \abs{k\of{1,s}}^{1 + \epsilon} \dd{s}}^{\frac{1 + \delta}{1 + \epsilon}} 
			\of{\int_0^1 s^{\frac{\delta}{\epsilon - \delta}\of{1 - \gamma}\of{1 + \epsilon}} \dd{s}}^{\frac{\epsilon - \delta}{1 + \epsilon}} 
			< 
			\infty 
			\period
		\end{align*}
	Therefore, the kernel $\kappa$ satisfies $\kappa\of{1,\cdot} \in L^{1 + \delta}\of{0,1}$. 
	Using the substitution $s^{\frac{1}{\gamma}}$ once more shows 
		\begin{align*}
			\int_0^1 \kappa\of{1,s} s^{\of{\gamma^{-1} + 1}n} \dd{s} 
			= 
			\int_0^1 k\of{1,s} s^{\of{\gamma + 1}n} \dd{s} 
		\end{align*}
	for all $n \in \N$. 
	Thus, the function $\Phi_\kappa$ exists as in~\eqref{eq:def_Phi} and coincides with $\Phi_k$.
	Then by Theorem~\ref{result:evoEq} and the previous Lemma the function $u$ given by~\eqref{eq:u_cor2} solves the evolution equation
		\begin{align*}
			u\of{t,x} 
			= 
			u_0\of{x} + \int_0^t \gamma \kappa\of{t^{\gamma},s^{\gamma}} s^{\gamma - 1} 
			\productspacing
			\frac{1}{2} \pardiff[2]{x} u\of{s,x} \dd{s} \comma \largespace 
			t > 0  \comma \largespace 
			x \in \R 
			\period
	\end{align*}
	This already proves the statement as $		\gamma \kappa\of{t^{\gamma},s^{\gamma}} s^{\gamma - 1} = k\of{t,s}$. 
\end{proof}

\begin{example}[Superstatistical Fractional Brownian Motion] \label{example:RP1}
	Let $x_0>0, \nu > 1/2$, $\rho > 1$.  
	Let $A_{x_0 , \nu , \rho}$ be a random variable with \emphasis{generalized gamma distribution}, i.e. with  the probability density function
		\begin{align*}
			f_{x_0,\nu,\rho}\of{a} 
			=
			\frac{\rho}{x_0^\nu \Gamma\of{\nu / \rho}} a^{\nu - 1} \e{-\of{a / x_0}^\rho} 
			\productspacing
			\1_{\of{0,\infty}}\of{a} 
			\period
		\end{align*}
	A general discussion of the generalized gamma distribution can be found in~\cite{S62}. 
	We note that, for our choice of parameters $x_0, \nu, \rho$, the Laplace transform $\laplace\bra{A_{x_0,\nu,\rho}}$ is an entire function and hence satisfies our assumptions. 
	The process $\sqrt{A_{x_0,\nu,\rho}}B^H_t$, $t \geq 0$, $H \in \of{0,1}$, where the fractional Brownian motion $B^H$ is independent from $A_{x_0,\nu,\rho}$, is called \emphasis{superstatistical fractional Brownian motion} with parameters $x_0, \nu, \rho$  (cf.~\cite{RVP22}). 
	To keep in line with~\cite{RVP22} we consider $X \coloneqq \sqrt{A_{2x_0,\nu,\rho}} B^H$, 
	i.e. a superstatistical fractional Brownian motion with parameters $2x_0$, $\nu$ and $\rho$.
	For each $t > 0$, $X_t$  has a probability density function $P^{2x_0, \nu, \rho}_{2H}\of{\cdot,t}$ given by 
		\begin{align*}
			P^{2x_0, \nu, \rho}_{2H}\of{t,x} 
			= 
	 		\int_0^\infty \of{4 \pi a t^{2H}}^{-1 / 2} \e{-\frac{x^2}{4at^{2H}}} 
			\productspacing
			f_{2x_0,\nu,\rho}\of{a} \dd{a}.
		\end{align*}
	Furthermore, one can see that $P^{2x_0, \nu, \rho}_{2H}\of{\cdot,x}$ is continuously differentiable for all $x \neq 0$. 

\begin{lemma} \label{lem:RunfPagn}
	The following identity is true for any $\delta > 0$, $x_0 > 0$, $\nu > 1/2$ and $\rho > 1$ such that $\frac{\nu}{\rho} - 1 < \frac{\delta}{2}$:
		\begin{align}\label{eq:RunfPagn}
			P^{2x_0, \nu+1, \rho}_{\delta}\of{t,x}
			=
			\frac{\Gamma\of{\nu / \rho}}{\Gamma\of{\of{\nu+1} / \rho}}
			\productspacing			
			_tD^{\frac{\nu}{\rho}-1, \frac{1}{\rho}}_{\delta \rho}
			\productspacing			
			P^{2x_0, \nu, \rho}_{\delta}\of{t,x} \comma \largespace 
			t > 0 \comma \largespace 
			x \in \R
			\comma
\end{align}
	where $_tD^{\gamma,\mu}_\eta$, $\mu,\eta>0$, $\gamma\in\R$, is the Erd\'elyi-Kober fractional operator with respect to time variable $t$ (cf. formula~(3.14) in~\cite{Lu4ko-Kiryakova}), i.e.
		\begin{align}\label{eq:EKoperator}
			_tD^{\gamma,\mu}_\eta\ffi\of{t} 
			=
			\frac{1}{2 \pi i}\int_{\gamma - i\infty}^{\gamma + i\infty} \frac{\Gamma\of{1 + \gamma + \mu - \of{s / \eta}}}{\Gamma(1+\gamma-(s/\eta))}
			\productspacing			
			\mellin_t\bra{\ffi}\of{s}
			\productspacing
			t^{-s} \dd{s}.
		\end{align}
\end{lemma}
	The Mellin transform $\mellin_t$ will be clarified in Appendix below. The proof of the formula~\eqref{eq:RunfPagn} is outlined in~\cite{RVP22} (see formulas (3.9) and (4.3) in~\cite{RVP22}). 
	For the convenience of the reader, we present a more detailed proof in Appendix below.
	Now let $\nu\of{t} \coloneqq Ct^\beta$, with $\beta \geq 0$ and
  		\begin{align*}
  	  		C
    		=
    		C_{H,\beta}^{-H}
   		=
    		\of{\frac{2H\of{2H-1}}{2\beta + 2H} 
    		\productspacing
    		\mathrm{B}\of{\beta + 1, 2H - 1}}^{-1/2}
    		\period
  		\end{align*}
	Furthermore let $H > \frac{1}{2}$ if $\beta > 0$ and let $H \in \of{0,1}$ if $\beta = 0$. 
	Let $\delta \coloneqq 2H + 2\beta$. 
	By Lemma~\ref{result:timeChange} we have $v_\nu\of{t} = t^\delta$. 
	Then $Z_t \coloneqq \int_0^t \nu\of{s} \dd{X_s}$, $t \geq 0$,
is in $\mathfrak{DI}_\infty\of{X}$. 

	\begin{proposition}\label{result:prop1}
		In the above setting, let $u_0\in \Sch\of{\R}$. 
		The function $u$ given by~\eqref{eq:def_u} solves the following evolution equation 
			\begin{align}\label{eq:evEq2}
				u\of{t,x)}
				=
				u_0\of{x} + \delta x_0 \int_0^t t^{\delta-1} \pardiff[2]{x} \bra{_sD^{\frac{\nu}{\rho}-1,\,\frac{1}{\rho}}_{\delta \rho} u\of{s,x}} \dd{s}
				\period
			\end{align}
	\end{proposition}

	\begin{proof}
		Let us calculate $\tilde{u}$ as in formula~\eqref{eq:u-tilde}. 
		By formula~\eqref{eq:RunfPagn} we have
			\begin{align*}
				\tilde{u}\of{t,x}
				&=
				\E\bra{Au_0\of{x + Z_t}}
				\\ &=
				\int_\R \int_0^\infty a u_0\of{x + \sqrt{a} y} \of{2\pi t^{\delta}}^{-1/2}\exp\of{-\frac{y^2}{2t^\delta}} 
				\productspacing				
				f_{2x_0,\nu,\rho}\of{a}\dd{a} \dd{y}
				\\ &= 
				\int_\R u_0\of{x + z} \of{\int_0^\infty \of{2\pi a t^{\delta}}^{-1/2}\exp\of{-\frac{z^2}{2at^\delta}} 
				\productspacing
				a f_{2x_0,\nu,\rho}\of{a}} \dd{z}
				\\ &=
				\int_\R u_0\of{x+z} \frac{2x_0 \Gamma\of{\of{\nu+1} / \rho}}{\Gamma\of{\nu / \rho}} 
				\productspacing
				P^{2x_0,\nu+1,\rho}_{\delta}\of{t,z} \dd{z}
				\\ &= 
				2x_0 \productspacing _tD^{\frac{\nu}{\rho}-1, \frac{1}{\rho}}_{\delta \rho} u\of{t,x}. 
			\end{align*}
		The statement then follows from formula~\eqref{eq:itoApp}.
	\end{proof}

	Note that in the case $\beta=0$ equation~\eqref{eq:evEq2} implies equation~(4.5) of~\cite{RVP22} for the probability density function $P^{2x_0,\nu+1,\rho}_{2H}$.
\end{example}

\begin{example}
	In the setting of Example~\ref{example:RP1} consider the process $\ZZ_t \coloneqq \exp\of{Z_t}$. Let $u_0 \in \Sch\of{\R}$. 
	Consider
		\begin{align}\label{eq:myZ}
			u\of{t,x} 
			\coloneqq
			\E\bra{u_0\of{x\ZZ_t}} \comma \largespace
			t \geq 0 \comma \largespace 
			x \in \R.
		\end{align}
	Using the It\^o formula as in Theorem~\ref{result:itoFormula} and taking expectation  we obtain 
		\begin{align*}
			u\of{t,x}
			&=
			u_0\of{x} + \frac{\delta}{2}\int_0^t s^{\delta-1} \E\bra{A\of{ u^{\prime\prime}_0\of{x \ZZ_s} \of{x \ZZ_s}^2 + u^{\prime}_0\of{x \ZZ_s} x \ZZ_s}} \dd{s}
			\\ &=
			u_0\of{x} + \frac{\delta}{2} \int_0^t s^{\delta-1} \E\bra{A\of{x^2 \pardiff[2]{x} u_0\of{x\ZZ_s} + x \pardiff{x} u_0\of{x\ZZ_s}}} \dd{s}
			\\ &=
			u_0\of{x} + \frac{\delta}{2} \int_0^t s^{\delta-1} \of{x^2 \pardiff[2]{x} \E\bra{A u_0\of{x\ZZ_s}} +x \pardiff{x} \E\bra{Au_0\of{x\ZZ_s}}} \dd{s}
			\period
		\end{align*}
	In a similar way as in the proof of Proposition~\ref{result:prop1}, we obtain by formula~\eqref{eq:RunfPagn} that $u$, given by~\eqref{eq:myZ}, satisfies the evolution equation
		\begin{align*}
			u\of{t,x}
			=
			u_0\of{x} + \delta x_0 \int_0^t s^{\delta-1} \of{x^2 \pardiff[2]{x}
			\productspacing 
			_sD^{\frac{\nu}{\rho}-1, \frac{1}{\rho}}_{\delta \rho}u\of{s,x} + x \pardiff{x} 
			\productspacing 
			_sD^{\frac{\nu}{\rho}-1,\,\frac{1}{\rho}}_{\delta \rho}u\of{s,x}} \dd{s}
			\period
		\end{align*}
\end{example}


\section*{Appendix}
Let us prove Lemma~\ref{lem:RunfPagn}. 
Consider the \emphasis{Kr\"atzel function} with parameters $\nu > 0$, $\rho > 0$:
	\begin{align*}
		Z^\nu_\rho\of{u}
		\coloneqq
		\int_0^\infty \lambda^{\nu-1} \e{-\frac{u}{\lambda}} \e{-\lambda^\rho}\dd{\lambda}
		\comma \largespace u > 0
		\period
	\end{align*}
Then the function 
	\begin{align}\label{KraetzelPDF}
		\widetilde{f}^\nu_\rho\of{x}
		\coloneqq
		\frac{\rho}{\Gamma\of{\frac{\nu+1}{\rho}}} Z^\nu_\rho\of{x} \1_{\of{0,\infty}}\of{x}
	\end{align}
is a probability density function, cf.~\cite{Princy}.
And we have
	\begin{align*}
		P^{2x_0,\nu,\rho}_\delta\of{t,x}
		&=
		\int_0^\infty \frac{1}{\sqrt{4 \pi a t^{\delta}}} \e{-\frac{x^2}{4at^{\delta}}} 
		\productspacing		
		f_{2x_0,\nu,\rho}\of{a} \dd{a}
		\\ &=
		\frac{\rho}{\Gamma\of{\frac{\nu}{\rho}}\sqrt{4\pi x_0 t^\delta}}Z^{\nu-1/2}_\rho\of{\frac{x^2}{4x_0t^\delta}}
		\period
	\end{align*}
For a sufficiently good function $f : \left[0,+\infty\right) \to \R$, the \emphasis{Mellin integral transform} is defined as
	\begin{align*}
		\mellin_t\bra{t}\of{s}
		\coloneqq
		\int_0^{\infty} f\of{t} \productspacing t^{s-1} \dd{t}
		\comma
	\end{align*}
and the \emphasis{inverse Mellin integral transform} as
	\begin{align*}
		f\of{t}
		=
		\frac{1}{2\pi i}\int_{\gamma-i\infty}^{\gamma+i\infty} \of{\mellin_t\bra{t}\of{s}} \productspacing t^{-s}ds,
	\end{align*}
for some suitable $\gamma \in \R$ (see~\cite{Lu4ko-Kiryakova} for the details). 
In particular, it is known (cf.~\cite{Princy}) that the Mellin transform of function $\widetilde{f}^\nu_\rho$ in~\eqref{KraetzelPDF} is given for any $\nu,\rho > 0$ and $s$ such that $\RE s > 0$  by 
	\begin{align*}
		\mellin_t\bra{\widetilde{f}^\nu_\rho}\of{s}
		=
		\frac{\Gamma\of{s}\Gamma\of{\of{s+\nu} / \rho}}{\Gamma\of{\of{\nu+1} / \rho}}
		\period
	\end{align*}
Therefore, in the case $\delta > 0$, $x_0 > 0$, $\rho > 1$, $\nu > 1/2$ and $\frac{\nu}{\rho} - 1 < \frac{\delta}{2}$, we have for $s$ such that $\RE s < \delta / 2$
	\begin{align*}
		\mellin_t &\bra{P^{2x_0, \nu, \rho}_\delta\of{t,x}}\of{s}
		\\ &=
		\int_0^{\infty}\frac{\rho}{\Gamma\of{\nu/\rho}} \of{4\pi x_0 t^\delta}^{-1/2}Z^{\nu-1/2}_\rho\of{\frac{x^2}{4x_0t^\delta}}t^{s-1} \dd{t}
		\\ &=
		\frac{\Gamma\of{\of{\nu + 1/2}/ \rho}}{\delta\Gamma\of{\nu/\rho}} \of{4\pi x_0}^{-1/2} \of{\frac{4x_0}{x^2}}^{\frac{1}{2} - \frac{s}{\delta}}
		\int_0^{\infty}\widetilde{f}^{\nu-1/2}_\rho\of{\tau}\tau^{-\frac{s}{\delta}-\frac{1}{2}} \dd{\tau}
		\\ &=
		\frac{\Gamma\of{\of{\nu+1/2}/\rho}}{\delta\Gamma\of{\nu/\rho}} \of{4\pi x_0}^{-1/2} \of{\frac{4x_0}{x^2}}^{\frac{1}{2} - \frac{s}{\delta}}
		\mellin_t\bra{\widetilde{f}^{\nu-1/2}_\rho} \of{-\frac{s}{\delta}+\frac{1}{2}}
		\\ &=
		\frac{\Gamma\of{\frac{1}{2} - \frac{s}{\delta}} \Gamma\of{\frac{\nu}{\rho}-\frac{s}{\rho\delta}}}{\delta\sqrt{4\pi x_0} \mediumspace \Gamma\of{\nu/\rho}}\of{\frac{x^2}{4x_0}}^{\frac{s}{\delta}-\frac{1}{2}}
		\period
	\end{align*}
Hence, the right hand side of formula~\eqref{eq:RunfPagn}, that is
	\begin{align*}
		\frac{\Gamma\of{\nu / \rho}}{\Gamma\of{\of{\nu+1}/\rho}} \,_tD^{\frac{\nu}{\rho} - 1, \frac{1}{\rho}}_{\delta \rho} P^{2x_0, \nu, \rho}_{\delta}\of{t,x}
		\comma
	\end{align*}
has the following view:
	\begin{align*}
		&\frac{\Gamma\of{\nu/\rho}}{\Gamma\of{\of{\nu+1}/\rho}}\frac{1}{2\pi i}\int\limits_{\frac{\nu}{\rho} - 1 - i\infty}^{\frac{\nu}{\rho} - 1 + i\infty} \frac{\Gamma\of{\frac{\nu + 1 - s/\delta}{\rho}}}{\Gamma\of{\frac{\nu - s/\delta}{\rho}}} \mellin_t\bra{P^{2x_0, \nu, \rho}_{\delta} \of{t,x}}\of{s} \productspacing t^{-s} \dd{s}
		\\&=
		\frac{1}{2\pi i \delta\Gamma\of{\of{\nu+1}/\rho} \sqrt{4\pi x_0}}\int\limits_{\frac{\nu}{\rho} - 1 - i\infty}^{\frac{\nu}{\rho} - 1 + i\infty}\Gamma\of{\frac{\nu+1-\frac{s}{\delta}}{\rho}} \Gamma\of{\frac{1}{2} - \frac{s}{\delta}} \of{\frac{z^2}{4x_0}}^{\frac{s}{\delta}-\frac{1}{2}} t^{-s} \dd{s}
		\period
\end{align*}
On the other hand, the left hand side of formula~\eqref{eq:RunfPagn} has the following view 
	\begin{align*}
		&P^{2x_0, \nu+1, \rho}_{\delta}\of{t,x}
		=
		\frac{1}{2\pi i}\int_{\gamma-i\infty}^{\gamma+i\infty}\mellin_t\bra{P^{2x_0, \nu+1, \rho}_{\delta}\of{t,x}}\of{s} t^{-s} \dd{s}
		\\ &=
		\frac{1}{2\pi i \delta\Gamma\of{\of{\nu + 1} / \rho} \sqrt{4\pi x_0}}\int\limits_{\gamma - i\infty}^{\gamma + i\infty} \Gamma\of{\frac{\nu + 1 - \frac{s}{\delta}}{\rho}} \Gamma\of{\frac{1}{2}-\frac{s}{\delta}} \of{\frac{x^2}{4x_0}}^{\frac{s}{\delta}-\frac{1}{2}}t^{-s} \dd{s}
		\period
	\end{align*}
where one can take any $\gamma < \delta / 2$, in particular $\gamma \coloneqq \frac{\nu}{\rho} - 1$. 
Therefore, the right hand side and the left hand side of~\eqref{eq:RunfPagn} coincide.


\printbibliography

\end{document}